\documentclass[reqno,10pt]{amsart}

\usepackage{url}
\usepackage{xcolor}
\definecolor{darkgreen}{rgb}{0,0.5,0}

\usepackage[T1]{fontenc}
\usepackage{amsmath,amsfonts,amssymb,amsthm}
\usepackage{mathtools}
\usepackage{comment}
\usepackage{pgf,tikz}
\usetikzlibrary{matrix,arrows,calc}
\usepackage{tikz-cd}
\usepackage{enumitem}
\usepackage[colorinlistoftodos]{todonotes}
\usepackage{cancel}
\usepackage{lmodern}
\usepackage{aligned-overset}
\usepackage{adjustbox}
\usepackage{footnote}
\makesavenoteenv{tabular}

\usepackage[
	backend=bibtex,  %
	style=alphabetic,
	minalphanames=3,
	maxnames=99,
	maxalphanames=4,
	giveninits=true,
	isbn=false,
	doi=true,
]{biblatex}
\usepackage[
	colorlinks, 
	citecolor=darkgreen,
]{hyperref}
\usepackage{cleveref}

\newcommand{\sm}{\mathrm{sm}}

\numberwithin{equation}{section}

\newtheorem{thmABC}{Theorem}

\newtheorem{thm}{Theorem}

\theoremstyle{remark}
\newtheorem{rem}[thm]{Remark}

\newtheorem{example}[thm]{Example}

\theoremstyle{definition}
\newtheorem{defn}[thm]{Definition}

\numberwithin{thm}{section}

\AddToHook{env/defn/begin}{\crefalias{thm}{defn}}
\AddToHook{env/prop/begin}{\crefalias{thm}{prop}}
\AddToHook{env/lemma/begin}{\crefalias{thm}{lemma}}
\AddToHook{env/conjecture/begin}{\crefalias{thm}{conjecture}}
\AddToHook{env/cor/begin}{\crefalias{thm}{cor}}
\AddToHook{env/rem/begin}{\crefalias{thm}{rem}}
\AddToHook{env/remarks/begin}{\crefalias{thm}{remarks}}
\AddToHook{env/example/begin}{\crefalias{thm}{example}}
\AddToHook{env/constr/begin}{\crefalias{thm}{constr}}

\newcommand{\cB}{\mathcal B}
\newcommand{\bC}{\mathbb C}

\newcommand{\cC}{\mathcal C}

\newcommand{\cpt}{\mathrm{cpt}}
\newcommand{\rd}{\mathrm d}

\newcommand{\cF}{\mathcal F}

\newcommand{\bG}{\mathbb G}

\newcommand{\length}{\mathrm{length}}

\newcommand{\rH}{\mathrm H}

\newcommand{\fp}{\mathfrak p}
\newcommand{\fq}{\mathfrak q}

\newcommand{\bP}{\mathbb P}

\newcommand{\cP}{\mathcal P}
\newcommand{\bQ}{\mathbb Q}
\newcommand{\cQ}{\mathcal Q}

\newcommand{\bZ}{\mathbb Z}
\newcommand{\cO}{\mathcal{O}}
\newcommand{\cX}{\mathcal{X}}
\newcommand{\cA}{\mathcal{A}}
\newcommand{\cY}{\mathcal{Y}}

\newcommand{\cG}{\mathcal{G}}

\newcommand{\bR}{\mathbb{R}}

\newcommand{\fS}{\mathfrak{S}}

\newcommand{\Tr}{\mathrm{Tr}}

\newcommand{\csp}{{\mathrm{csp}}}

\renewcommand{\div}{\mathrm{div}}

\newcommand{\Kbar}{\overline{K}}
\newcommand{\Kpbar}{\overline{K}_{\fp}}

\DeclareMathOperator{\AJ}{AJ}

\DeclareMathOperator{\cD}{\mathcal{D}}

\DeclareMathOperator{\Div}{Div}

\DeclareMathOperator{\bF}{\mathbb{F}}

\DeclareMathOperator{\dR}{dR}

\DeclareMathOperator{\Jac}{Jac}

\DeclareMathOperator{\Lie}{Lie}

\DeclareMathOperator{\Res}{Res}

\DeclareMathOperator{\Hom}{Hom}

\DeclareMathOperator{\rank}{\mathrm{rk}}

\DeclareMathOperator{\Sel}{Sel}
\DeclareMathOperator{\Spec}{Spec}

\DeclareMathOperator{\ord}{ord}

\title{Affine Chabauty II}
\author{Marius Leonhardt}
\address{Marius Leonhardt,
	Department of Mathematics and Statistics,
	Boston University,
	665 Commonwealth Ave,
	Boston, MA 02215,
	USA}
\email{mleonhar@bu.edu}

\author{Martin Lüdtke}
\address{Martin Lüdtke,
	Institut für Mathematik,
	Carl von Ossietzky Universität Oldenburg,
	26111 Oldenburg,
	Germany
}
\email{martin.luedtke@uol.de}

\begin{document}
	
	\thispagestyle{empty}
	
	\begin{abstract}
		We present an algorithm for determining the set of $S$-integral points on an affine curve based on the Affine Chabauty method developed in the first part of this series.
		We achieve this by constructing explicit logarithmic differentials whose integrals take on prescribed values on $S$-integral points.
		Along the way, we prove a $p$-adic residue theorem for Coleman integrals of log differentials.
	\end{abstract}
	
	\maketitle
	\tableofcontents

\section{Introduction}
\label{sec: introduction}

Let $K$ be a number field and $S$ a finite set of primes of $\mathcal{O}_K$.
Let $Y/K$ be a smooth affine curve and let $\cY/\cO_{K,S}$ be a
regular model of $Y$ over the ring of $S$-integers of $K$.
If $Y$ is hyperbolic, the set of $S$-integral points $\cY(\cO_{K,S})$ is finite by the theorems of Siegel, Mahler, and Faltings.
However, finding the points in $\cY(\cO_{K,S})$ remains a difficult open problem in general.
In \cite{affchab1} we presented a new approach addressing this question for curves satisfying the hypothesis
\begin{equation}
\label{eq:chabauty-condition-over-number-field-intro}
r + \# S + ([K : \bQ]-1)n < g + \rank \cO_K^\times + \#|D| + n_2(D) - 1,
\end{equation}
an inequality involving the genus $g$, the Mordell--Weil rank $r$, the cardinality of~$S$, the degree and unit rank of $K$, and invariants $n=\# D(\overline{K})$ and $n_2(D) = \frac12\# (D(\bC) \smallsetminus  D(\bR))$ coming from the boundary $D$ of $Y$.
Our method is an $S$-integral analogue of the method of Chabauty--Coleman \cite{Cha41, Col85}, and a refinement of Skolem's method \cite{skolem} in the genus~0 case.
For an auxiliary prime $\fp \not\in S$, the set $\cY(\cO_{K,S})$ is contained in a finite union of subsets of $\cY(\cO_{\fp})$, each of which is defined by the $\fp$-adic integral of a logarithmic differential (i.e.\ meromorphic differential on the compactification of~$Y$ with at worst simple poles at $D$) taking on a prescribed value.
In this article we present an algorithm to determine these log differentials and the possible values of their integrals in practice, thus computing a finite subset of $\cY(\cO_{\fp})$ containing $\cY(\cO_{K,S})$.

\subsection{Main results}

Let $Y/K$ be given as $X\smallsetminus D$ where $X/K$ is a smooth projective curve and $D \neq \emptyset$ is a finite, reduced, closed subscheme.
The closed points of $D$ are called \emph{cusps} and we denote the underlying set of $D$ by $|D|$. 
Let $S$ be a finite set of primes of~$K$.
Let~$\cX$ be a regular model of $X$ over the ring~$\cO_{K,S}$ of $S$-integers, i.e.\ a regular, flat, projective $\cO_{K,S}$-scheme with an isomorphism $\cX_{K} \cong X$ \cite[Definition~10.1.1]{liu2006algebraic}.
Let $\cD$ be the closure of~$D$ in~$\cX$ and set $\cY \coloneqq \cX \smallsetminus \cD$. Assume that $Y(K) \neq\emptyset$ and fix a base point $P_0 \in Y(K)$. 
We use the following notation, which will be kept throughout this paper:
\begin{itemize}
	\item $r \coloneqq \rank J(K)$ the Mordell--Weil rank of the Jacobian $J$ of~$X$;
	\item $g$ the genus of~$X$;
	\item $\#|D| > 0$ the number of cusps;
	\item $n \coloneqq \#D(\overline{K})>0$ the number of geometric cusps;%
	\item $[K:\bQ]n = n_1(D) + 2n_2(D)$ with $n_1(D) \coloneqq \#D(\bR)$ the number of real cusps and $n_2(D) = \frac12\# (D(\bC) \smallsetminus  D(\bR))$ the number of conjugate pairs of complex cusps.
	Here $D$ is viewed as a scheme over $\bQ$ (rather than over $K$).
\end{itemize}

We partition the $S$-integral points
\[
	\cY(\cO_{K,S}) = \coprod_{\Sigma} \cY(\cO_{K,S})_{\Sigma},
\]
where $\Sigma$ runs through the finitely many $S$-integral \emph{reduction types} of $\cY$, see \Cref{sec:redtypes}. Roughly speaking, the reduction type of an $S$-integral point specifies, for each prime~$\ell$, the component of the mod-$\ell$ fibre or, if $\ell \in S$, possibly the cusp onto which the point reduces.
In \cite[Proposition 3.13]{affchab1} we showed that the Abel--Jacobi image of $\cY(\cO_{K,S})_{\Sigma}$ inside the generalised Jacobian $J_Y(K)$ is contained in the Selmer set $\Sel(P_0,\Sigma)$, which is empty or a translate of a finitely generated abelian group of controllable rank and is defined using arithmetic intersection theory and the $D$-intersection map $\sigma\colon J_Y(K) \to V_D$, see \Cref{sec:affchab1}.
Let $\omega_1,\ldots,\omega_{g+n-1}$ be a basis of the space $\rH^0(X,\Omega^1(D))$ of logarithmic differentials on $(X,D)$ with $\omega_1,\ldots,\omega_g$ holomorphic. Assume $n \geq 2$. 
Our main result is:

\begin{thmABC}\label{thm:A}
	Let $\Sigma$ be an $S$-integral reduction type with cuspidal part $\Sigma^{\csp}$ and let $\fp \not\in S$ be a prime of $K$.
	Suppose $(a_1,\ldots,a_{g+n-1}) \in K_{\fp}^{g+n-1}$ is contained in the kernel of the matrix $M(\Sigma^{\csp})$ defined in \eqref{eq:MSigma}--\eqref{eq:matrixD}, let $\omega \coloneqq a_1 \omega_1 + \dots + a_{g+n-1} \omega_{g+n-1}$, and let $c \coloneqq c(P_0,\Sigma,\omega)$ be the constant defined by
	\eqref{eq:constants} and \eqref{eq:b-lambda}.
	Then
	\[
		\int_{P_0}^P \omega = c \quad \text{ for all } P \in \cY(\cO_{K,S})_\Sigma.
	\]
\end{thmABC}

Thus $\cY(\cO_{K,S})_{\Sigma}$ is contained in the set of zeros of the $\fp$-adic analytic function $\cY(\cO_{\fp}) \to K_{\fp}$ that sends $P$ to $\int_{P_0}^P \omega - c$.
The constant $c(P_0,\Sigma,\omega)$ and the entries of the matrix $M(\Sigma^{\csp})$ are given by explicit formulas in terms of $\fp$-adic integrals of the basis differentials $\omega_j$, the residues of the $\omega_j$ at the cusps, the $\fp$-adic logarithm of certain elements of the residue fields of the cusps, and intersection numbers on the regular model~$\cX$.

To prove \Cref{thm:A}, let $\log_{J_Y}\colon J_Y(K_{\fp}) \to \rH^0(X_{K_{\fp}},\Omega^1(D))^\vee$ be the logarithm of the $\fp$-adic Lie group $J_Y(K_{\fp})$, which we restrict to the direction $W$ of $\Sel(P_0,\Sigma)$.
We show that the matrix $M(\Sigma^{\csp})$ represents the dual map $\log_{J_Y}^\sharp \colon \rH^0(X_{K_{\fp}}, \Omega^1(D)) \to W^\vee\otimes_{\bQ} K_{\fp}$ with respect to a certain explicit basis of $W$.
This basis is constructed in \eqref{eq:F-and-Gi}, \eqref{eq:f-of-Q} and \eqref{eq:t-ell} and consists of dual bases of the Mordell--Weil group $J(K)$ of the Jacobian of $X$, of the unit groups at the cusps, and of a basis of $\sigma(W)=U$.
For $F\in W$ and $\omega$ a logarithmic differential, the fundamental calculation of $\log_{J_Y}^\sharp(\omega)(F) = \int_F \omega$ reduces to calculating $\int_{\div(f)} \omega$ for a certain rational function~$f$.
A $p$-adic residue theorem, proved in \Cref{sec:integration}, expresses $\int_{\div(f)} \omega$ in terms of the residues of $\omega$ and the values of $f$ at the cusps, and the latter can be recovered from $\sigma(F)$ up to units.
The main insight is how to combine the $D$-intersection map $\sigma$ with $\fp$-adic integration.

\subsection{Affine Chabauty algorithm}
\label{sec:algorithm}

\Cref{thm:A} is the central step in building an algorithm based on the Affine Chabauty method.
That is, on input an affine curve $Y=X\smallsetminus D$, a finite set of primes $S$ satisfying \eqref{eq:chabauty-condition-over-number-field-intro}, a regular model $\cY = \cX \smallsetminus \cD$ over $\cO_{K,S}$, and an auxiliary prime $\fp \not\in S$, we want to determine a finite subset of $\cY(\cO_{\fp})$ containing $\cY(\cO_{K,S})$.
Let us give a sketch of how such an algorithm proceeds:
\begin{enumerate}[label=(\arabic*)]
	\item \label{item:algo-D-transversal-model}
	Extend $\cX$ to a regular model over~$\cO_K$. Ensure that the model is $D$-transversal over primes in~$S$ via blow-ups if necessary (see Section~\ref{sec:redtypes}).
	\item \label{item:algo-intersection-numbers}
	For each prime~$\fq$ of bad reduction for~$\cX$, determine the set of components and the intersection matrix of the mod-$\fq$ fibre~$\cX_{\fq}$. 
	\item \label{item:algo-mw-basis}
	Find divisors $G_i\in \Div^0(Y)$ generating a full rank subgroup of the Mordell--Weil group $J(K)$.
	\item \label{item:algo-correction-divisors}
	Using \cite[Lemma~2.7]{affchab1}, determine the vertical correction divisors $\Phi_{\fq}(G_i)$ for each prime $\fq$ of $K$. This involves the Moore--Penrose pseudoinverse of the intersection matrix of the mod-$\fq$ fibre computed in step~\ref{item:algo-intersection-numbers}.
	\item \label{item:algo-log-differentials}
	Find a basis $\omega_1,\ldots,\omega_{g+n-1}$ of the space of log differentials $\rH^0(X,\Omega^1(D))$ and compute their residues $\Res_Q(\omega_j) \in k(Q)$ at all cusps $Q \in |D|$.
	\item \label{item:algo-units}
	For each $Q\in |D|$, find $\varepsilon_i \in \cO_{k(Q)}^\times$ generating a full rank subgroup of $\cO_{k(Q)}^\times$.
	\item \label{item:algo-prime-ideal-generators}
	For finitely many ``special'' primes $\fq$ of $K$, determine generators $\rho_{\fq}$ of $\fq$ in $K^\times\otimes {\bQ}$.
	Similarly, for $Q\in |D|$ and finitely many ``special'' primes $\lambda$ of $k(Q)$, determine generators $\pi_{\lambda}$ of $\lambda$ in $k(Q)^\times\otimes {\bQ}$.
	\item \label{item:algo-integrals}
	Find a method for calculating $\fp$-adic integrals of log differentials on $X$.
	\item \label{item:algo-chabauty-function}
	For each $S$-integral reduction type~$\Sigma$, determine the matrix $M(\Sigma^{\csp})$ and find a non-trivial element $(a_1,\ldots,a_{g+n-1})$ in the kernel. Compute the constant $c=c(P_0,\Sigma,\omega)$ for the  Chabauty differential $\omega = \sum_j a_j \omega_j$.
	\item \label{item:algo-chabauty-locus}
	On each residue disc, expand the $\fp$-adic analytic function $P \mapsto \int_{P_0}^P \omega$ into a power series and determine the locus where $\int_{P_0}^P \omega=c$.
\end{enumerate}

\Cref{thm:A} guarantees that the union of the sets obtained in step~\ref{item:algo-chabauty-locus} over all $S$-integral reduction types~$\Sigma$ is a finite subset of $\cY(\cO_{\fp})$ containing $\cY(\cO_{K,S})$. The central step and main focus of this article is~\ref{item:algo-chabauty-function}. The results of all prior steps enter in the computation of the matrices $M(\Sigma^{\csp})$ and the constants $c(P_0,\Sigma,\omega)$, as described in detail in Section~\ref{sec:explicit}.

In individual examples, calculating intersection numbers \ref{item:algo-intersection-numbers} and a basis of log differentials \ref{item:algo-log-differentials} can be done by hand.
In general, Magma's \texttt{RegularModel} functionality or, in the case of hyperelliptic curves, the cluster picture \cite{m2d2, usersguide} can be used to describe the special fibres  \ref{item:algo-intersection-numbers}. %
For hyperelliptic curves and the superelliptic curves considered in Section~\ref{sec:examples}, we are able to compute $\fp$-adic integrals of log differentials \ref{item:algo-integrals} by reducing them to integrals of the particular differential $x^g \, \rd x/y$ on various hyperelliptic curves of genus~$g$, for which we can use existing code \cite{BBK10, BB12}.
There are methods in Magma \cite{magma} and SageMath \cite{sage} for computing regular models, Mordell--Weil bases, unit groups and generators of primes ideals, as required in steps \ref{item:algo-D-transversal-model}, \ref{item:algo-intersection-numbers}, \ref{item:algo-mw-basis}, \ref{item:algo-units}, \ref{item:algo-prime-ideal-generators}.

\begin{rem}
	Other approaches to computing ($S$-)integral points on affine curves include Baker's method \cite{baker}. %
	While often effective, the appearing height bounds can be enormous and the calculations involve base changing to number fields of large degree. %
	In practice, one often combines different techniques, for example also using descent or the Mordell--Weil sieve.
	Our method's main advantage is its applicability to general affine curves without relying on equations of a special shape.
	In addition, it leads to good bounds on the size of $\cY(\bZ_S)$ as explained in \cite[Theorem C]{affchab1}, and is geometric in nature, relying on arithmetic intersection theory.
	Its limitations come from implementations of steps \ref{item:algo-D-transversal-model} to \ref{item:algo-integrals}.
	For example, future extensions of Coleman integration techniques \ref{item:algo-integrals} beyond hyperelliptic curves will immediately extend the scope of our method.
	
	To illustrate what we can already achieve, we determine in \Cref{thm:example-split} the integral points on an even degree hyperelliptic curve whose rational points are, to the best of our knowledge, not provably known.
	In \Cref{exa:superell} we moreover compute a finite set of $p$-adic points containing the $S$-integral points, with $\# S=1$, for an affine curve of the form $y^3= f(x)$.
\end{rem}

\subsection{Structure of the paper}

\Cref{sec:affchab1} reviews the Affine Chabauty method developed in \cite{affchab1}.
We define intersection numbers, the $D$-intersection map, reduction types, and the Selmer set.
In \Cref{sec:integration} we prove the residue theorem for $p$-adic integrals of logarithmic differentials.
In \Cref{sec:explicit} we define the matrix $M(\Sigma^{\csp})$ and prove the main result.
We illustrate our method with several examples in \Cref{sec:examples}.

\subsection*{Acknowledgements}

We thank the anonymous referees for their helpful comments.
The first author is supported by a Walter Benjamin Scholarship from the Deutsche Forschungsgemeinschaft (DFG, German Research Foundation), project number LE 5634/1, and also acknowledges support from the DFG through TRR 326 Geometry and Arithmetic of Uniformized Structures, project number 444845124. The second author is supported by a Minerva Fellowship of the Minerva Stiftung Gesellschaft für die Forschung mbH and also acknowledges support through a guest postdoc fellowship at the Max Planck Institute for Mathematics in Bonn and through an NWO Grant, project number VI.Vidi.192.106.

\section{The Affine Chabauty method}
\label{sec:affchab1}

\subsection{Strategy}

Here we recall the necessary background from \cite{affchab1}.
As before, let $X$ be a smooth proper curve over a number field $K$, let $D \neq \emptyset$ be a finite set of closed points and set $Y = X \smallsetminus D$.
Let $J_Y$ denote the generalised Jacobian of $Y$ and $T_D$ be its toric part, see \cite[Section 2.1]{affchab1}.
We assume that $n\geq 2$; 
otherwise $D$ consists of a single $K$-rational point, the generalised Jacobian~$J_Y$ coincides with the Jacobian~$J$ and one can use the classical Chabauty method for finding not just $S$-integral points on $\cY$ but all rational points on~$X$.
Let $S$ be a finite set of primes of $K$ and let $\cX$ be a regular model of~$X$ over $\cO_K$ that is $D$-transversal over primes in~$S$, see \Cref{sec:redtypes}.
Let~$\cD$ be the closure of~$D$ in~$\cX$ and set $\cY = \cX \smallsetminus \cD$.
Fix a base point $P_0 \in Y(K)$ and an auxiliary prime~$\fp \not\in S$.
Our aim is to determine the set $\cY(\cO_{K,S})$ of $S$-integral points on $\cY$.
We partition the $S$-integral points $\cY(\cO_{K,S}) = \coprod_{\Sigma} \cY(\cO_{K,S})_{\Sigma}$, where $\Sigma$ runs through the finitely many $S$-integral reduction types of $\cY$, see \Cref{sec:redtypes}.
By \cite[Proposition 3.13]{affchab1}, the Abel--Jacobi image $\AJ_{P_0}(\cY(\cO_{K,S})_\Sigma)$ inside $J_Y(K)$ lies in the Selmer subset $\Sel(P_0,\Sigma)$, which is empty or a translate of a finitely generated subgroup of $J_Y(K)$.
For the computations that follow, it is more convenient to work with (translates of) vector spaces rather than abelian groups, see \Cref{sec:selmerset}.
Let us explain the proof of the main result \cite[Theorem 5.1]{affchab1} in this setting:
we locate $\cY(\cO_{K,S})_{\Sigma}$ inside the $\fp$-adic points $\cY(\cO_{\fp})$ via the following commutative diagram, which is the rational version of \cite[(5.3)]{affchab1}:
\begin{equation}
	\label{eq:chabauty-diagram}
	\adjustbox{scale=.9,center}{%
	\begin{tikzcd}
	&[-15pt] \cY(\cO_{K,S})_{\Sigma} \dar["\AJ_{P_0}"] \rar[hook] & \cY(\cO_{K,S}) \dar["\AJ_{P_0}"] \rar[hook] & \cY(\cO_{\fp}) \dar["\AJ_{P_0}"] \drar["\int_{P_0}"]& \\
	b_0 + W_{\bQ}(\Sigma^{\csp}) \rar[equal] & \Sel_{\bQ}(P_0, \Sigma) \rar[hook] & J_Y(K)_{\bQ} \rar[hook] & J_Y(K_{\fp})_{\bQ} \rar["\log_{J_Y}"] & \rH^0(X_{K_{\fp}}, \Omega^1(D))^\vee %
	\end{tikzcd}
}
\end{equation}

If $\log_{J_Y}$ restricted to $W_{\bQ}(\Sigma^{\csp})$ is not surjective, then we can find a nonzero log differential $\omega = \omega(\Sigma^{\csp})$ annihilating $\log_{J_Y}(W_{\bQ}(\Sigma^{\csp}))$.
Pulled back via $\AJ_{P_0}$, we see that the function 
\[
	\rho\colon \cY(\cO_{\fp}) \to K_{\fp}, \quad P \mapsto \int_{P_0}^P \omega
\]
takes the same value $c=c(P_0,\Sigma,\omega)$ for all $P\in \cY(\cO_{\fp})$ satisfying $\AJ_{P_0}(P) \in \Sel_{\bQ}(P_0,\Sigma)$, so in particular for $P\in \cY(\cO_{K,S})_{\Sigma}$.
Non-surjectivity of $\log_{J_Y}|_{W_{\bQ}(\Sigma^{\csp})}$ is guaranteed if $\dim_{\bQ} W_{\bQ}(\Sigma^{\csp}) < \dim_{K_{\fp}} \rH^0(X_{K_{\fp}}, \Omega^1(D)) = g+n-1$, which is equivalent to
\begin{equation}
\label{eq:chabauty-condition}
r + \# C(\Sigma) + ([K : \bQ]-1)n < g + \rank \cO_K^\times + \#|D| + n_2(D) - 1
\end{equation}
and implied by \eqref{eq:chabauty-condition-over-number-field-intro} as $C(\Sigma)\subseteq S$.
In \Cref{sec:explicit} we explain how the differential $\omega$ and the constant $c$ can be computed in practice.
In applications, we will also explain how the $\fp$-adic points $P$ that satisfy $\rho(P)=c$ can be computed with finite $\fp$-adic precision.

\subsection{Intersection numbers}
\label{sec:intersection}

Here we use the intersection number as defined in \cite[(2.3)]{affchab1}.
Recall from \cite[§III.2]{lang:arakelov} the definition of intersection numbers on the arithmetic surface~$\cX$. If $E,F$ are effective divisors on~$\cX$ without common component and $x$ is a closed point of~$\cX$, the \emph{intersection number} of $E$ and $F$ at~$x$ is defined as
\[ i_x(E,F) \coloneqq \length_{\cO_{\cX,x}}(\cO_{\cX,x}/(f,g)), \]
where~$f$ and~$g$ are defining equations for $E$ and $F$ in a neighbourhood of~$x$.
The \emph{intersection cycle} of such divisors is defined as the zero-cycle
\[ E . F \coloneqq \sum_{x} i_x(E,F) [x] \in Z_0(\cX). \]
For a prime $\fq$ of $K$, we also have the $\fq$-intersection number
\begin{equation}
\label{eq:m-intersection-number}
i_{\fq}(E,F) \coloneqq \sum_{x|\fq} i_x(E,F) [k(x) : k],
\end{equation}
where the sum is over all closed points of~$\cX$ above $\fq$. 
We use a generalisation of intersection numbers that allows us to ``intersect'' divisors on $\cX$ with $\tilde{\cD}$, the normalisation $\pi\colon \tilde\cD \to \Spec \cO_K$ of $\cO_K$ in $\cO(D)$.
Let $\nu\colon \tilde \cD \to \cX$. 
For any closed point $x \in \tilde \cD$ and effective divisor~$E$ on~$\cX$ without common component with~$\cD$, we write
\begin{equation}
\label{eq:intersection-Dtilde}
i_x(E,\tilde \cD) \coloneqq \ord_x(\nu^* f)
\end{equation}
with $f$ a defining equation for~$E$ in a neighbourhood of $\nu(x)$.
We also have the intersection cycle
\[ E.\tilde \cD \coloneqq \sum_{x\in|\tilde\cD|} i_x(E,\tilde \cD) [x] \in Z_0(\tilde \cD). \]

We use these intersection numbers to define the \emph{$D$-intersection map} $\sigma$ as in \cite[\S 3.1]{affchab1}.
For any degree-0 divisor~$F$ on~$Y$ with horizontal extension~$\cF$ on~$\cX$ there exists a vertical $\bQ$-divisor $\Phi(F)$ on~$\cX$ such that for all primes~$\fq$ of~$\cO_K$, the divisor $\Psi(F) \coloneqq \cF + \Phi(F)$ has $\fq$-intersection number zero with all vertical divisors over~$\fq$.
We write $\Phi(F)=\sum_{\fq} \Phi_{\fq}(F)$ with $\Phi_{\fq}(F)$ a vertical divisor above $\fq$.
Define
\begin{equation}
\label{eq:global-sigma}
	\sigma\colon J_Y(K) \to V_D \coloneqq Z_0(\tilde \cD)/\pi^* Z_0(\Spec(\cO_K)) \otimes_{\bZ} \bQ, \quad F \mapsto \Psi(F).\tilde \cD. %
\end{equation}
This is well-defined and naturally decomposes as a direct sum of local $D$-intersection maps $\sigma_{\fq}$.
See also \cite[\S 2.4]{affchab1} on how $\Phi(F)$ can be computed.

\subsection{Reduction types}
\label{sec:redtypes}

In the definition of reduction types it is important that we assumed our regular model $\cX/\cO_K$ to be \emph{$D$-transversal} over the primes in $S$, see \cite[Definition 2.17]{affchab1}.
That is, for $\fq \in S$ and $x\in (\cX_{\fq}^{\sm}\cap\cD)(\bF_{\fq})$ the intersection number $i_x(\cX_{\fq},\tilde{\cD})$ is equal to $1$.
If this condition is not satisfied, one can use blow-ups to construct a $D$-transversal model dominating the given one, see \cite[Proposition~2.19]{affchab1}.

We recall the definition of reduction types from \cite[Definition 3.8]{affchab1}.
Namely, an \emph{$S$-integral reduction type} $\Sigma = (\Sigma_{\fq})_{\fq}$ with cuspidal support $S_0 \coloneqq C(\Sigma) \subseteq S$ consists of
\begin{itemize}
	\item for each $\fq\not\in S_0$ an irreducible component $\Sigma_{\fq}$ of $\cX_{\fq}^{\sm} \smallsetminus \cD_{\fq}$ containing at least one $\bF_{\fq}$-point;
	\item for each $\fq \in S_0$ an $\bF_{\fq}$-point $\Sigma_{\fq}$ on the special fibre of~$\cD$ at which~$\cX_{\fq}$ is smooth.
\end{itemize}
The tuple $\Sigma^{\csp}\coloneqq (\Sigma_{\fq})_{\fq\in S_0}$ is called the \emph{cuspidal part} of $\Sigma$.
The subset $\cY(\cO_{K,S})_{\Sigma} \subseteq \cY(\cO_{K,S})$ consists of those points reducing onto $\Sigma_{\fq}$ modulo $\fq$ for all $\fq$.

\subsection{Selmer sets}
\label{sec:selmerset}

Let $\Sigma$ be an $S$-integral reduction type with cuspidal part $\Sigma^{\csp}\coloneqq (\Sigma_{\fq})_{\fq\in S_0}$.
Define the subgroup $U(\Sigma^{\csp}) = \langle [\Sigma_{\fq}]~|~\fq \in S_0\rangle$ of $V_D$ of rank $\# S_0$ and the element $b = b(P_0,\Sigma) = \sum_{Q,\lambda} b_\lambda [\lambda] \in V_D$ by
\begin{equation}
\label{eq:b-lambda}
b_{\lambda} = -i_{\lambda}(\cP_0,\tilde\cQ) + i_{\lambda}(\Phi_{\fq}(\cpt(\Sigma_{\fq}) - \cpt_{\fq}(P_0)), \tilde\cQ),
\end{equation}
where $Q\in |D|$ with closure $\cQ$ in $\cX$, $\lambda \in \tilde \cQ$ is a prime ideal of $\cO_{k(Q)}$, $\cP_0$ is the closure of $P_0$ in $\cX$, $\cpt(\Sigma_{\fq})$ denotes the component of the mod-$\fq$ fibre~$\cX_{\fq}$ containing $\Sigma_{\fq}$, and $\cpt_{\fq}(P_0)$ denotes the component of the mod-$\fq$ fibre onto which the base point~$P_0$ reduces.
Let $\fS(P_0,\Sigma) \coloneqq b + U(\Sigma^{\csp})$, see \cite[Definitions 3.9, 2.14 and 2.30]{affchab1}.
By \cite[Definition 3.12]{affchab1}, the \emph{Selmer set} $\Sel(P_0,\Sigma)\subseteq J_Y(K)$ is defined as the preimage of $\fS(P_0,\Sigma)$ under the $D$-intersection map $\sigma$.
The Selmer subset $\Sel(P_0,\Sigma)$ is empty or a translate of a finitely generated subgroup of rank $n_1(D) + n_2(D) - \#|D| - \rank \cO_K^\times + r + \# S_0$ by \cite[Proposition 3.14]{affchab1}.

It will be more convenient to work with vector spaces.
Let $J_Y(K)_{\bQ} \coloneqq J_Y(K)\otimes_{\bZ}\bQ$.
We denote the unique extension of $\sigma$ to $J_Y(K)_\bQ \to V_D$ also by $\sigma$.
Let $U_{\bQ}(\Sigma^{\csp}) \coloneqq U(\Sigma^{\csp}) \otimes_{\bZ} \bQ$ and $\fS_{\bQ}(P_0,\Sigma)\coloneqq b + U_{\bQ}(\Sigma^{\csp})$.
Let $\Sel_{\bQ}(P_0,\Sigma) \coloneqq \sigma^{-1}(\fS_{\bQ}(P_0,\Sigma))$ be the \emph{rational Selmer set}.
As the restriction $\sigma\colon T_D(K)_{\bQ} \to V_D$ is surjective \cite[Lemma 3.4]{affchab1}, so is $\sigma\colon J_Y(K)_{\bQ} \to V_D$.
Choose any $b_0\in J_Y(K)_{\bQ}$ with $\sigma(b_0)=b$.
Then $\Sel_{\bQ}(P_0,\Sigma) = b_0 + \sigma^{-1}(U_{\bQ}(\Sigma^{\csp}))$ is an affine space of dimension $n_1(D) + n_2(D) - \#|D| - \rank \cO_K^\times + r + \# S_0$ whose direction $W_{\bQ}(\Sigma^{\csp}) \coloneqq \sigma^{-1}(U_{\bQ}(\Sigma^{\csp}))$ only depends on $\Sigma^{\csp}$.

\section{A residue theorem for $p$-adic integrals}
\label{sec:integration}

Let $K$ be a complete subfield of $\bC_p$ (such as~$\bQ_p$ or a finite extension thereof), let $X$ be a smooth proper curve of genus~$g$ over~$K$, let $D \neq \emptyset$ be a finite set of closed points of~$X$ and set $Y = X \smallsetminus D$. Denote by $n \coloneqq D(\Kbar)$ the number of geometric cusps. A \emph{logarithmic differential} on $(X,D)$ is a meromorphic differential on~$X$ which is everywhere regular except for possibly simple poles at the points of~$D$. The $K$-vector space of logarithmic differentials on~$(X,D)$ is $\rH^0(X, \Omega^1(D))$.

We are using $p$-adic integrals of log differentials as defined by Zarhin \cite{zarhin}; see \cite[\S 4]{affchab1} for an overview.
We choose the Iwasawa branch of the $p$-adic logarithm $\log\colon K^\times \to K$ given by $\log(p) = 0$.
With this choice, we extend the logarithm of the $p$-adic Lie group $J_Y(K)$, which a priori is only defined on an open subgroup, to all of $J_Y(K)$.
Assume that $Y(K) \neq \emptyset$.
Then the logarithm of~$J_Y$ can be viewed as a map
\begin{equation}
	\label{eq:log-JY}
	\log_{J_Y}\colon J_Y(K) \to \rH^0(X, \Omega^1(D))^\vee.
\end{equation}
This allows us to define, for any degree zero divisor~$F$ on $\Div^0(Y)$, representing an element of~$J_Y(K)$, and for any log differential~$\omega$ on $(X,D)$ the integral
\[ \int_F \omega \coloneqq \log_{J_Y}(F)(\omega) \in K. \]

Tiny integrals, that is $\int_P^Q \omega \coloneqq \int_{[Q]-[P]} \omega$ for $P,Q\in Y(K)$ in the same residue disc, can be computed by expanding $\omega$ as a power series in a uniformiser of the disc and formally integrating and evaluating this power series at the end points.

When $\omega$ is a holomorphic differential on~$X$, the integral of~$\omega$ over any principal divisor $\div(f)$ vanishes since the integral is pulled back from the Jacobian~$J$ where principal divisors are trivial. When $\omega$ has poles, this is no longer true. In this section we show a formula for $\int_{\div(f)} \omega$ involving a sum over the poles of~$\omega$ and the residues of~$\omega$ at those poles, reminiscent of the residue theorem in complex analysis. The formula is used in \Cref{sec:explicit} to obtain an explicit method for constructing Chabauty functions.

\begin{defn}
	\label{def:residue}
	Let $\omega$ be a logarithmic differential on $(X,D)$ and let $Q$ be a closed point in~$D$. Let $t \in k(X)^\times$ be a uniformiser at~$Q$. Then we can write $\omega = f \frac{\rd t}{t}$ with $f \in k(X)^\times$ regular at~$Q$. The \emph{residue} of~$\omega$ at~$Q$ is defined by
	\[ \Res_Q(\omega) \coloneqq f(Q) \in k(Q). \]
	This definition does not depend on the choice of uniformiser.
\end{defn}

\begin{rem}
	Different definitions of the residue at a closed point exist in the literature. The definition above is a special case of the one given in \cite[\href{https://stacks.math.columbia.edu/tag/0FMU}{Tag 0FMU}]{stacks-project}. In \cite[§II.7]{serrealggp}, residues of differentials on curves are defined over algebraically closed fields via Laurent series expansions. It specialises to our definition upon base changing from~$K$ to $\Kbar$ and choosing $K$-embeddings $k(Q) \hookrightarrow \Kbar$. In \cite{tate:residues}, a different definition of the residue at a closed point is given using traces of linear operators on the function field of~$X$. Tate's residues at a closed point~$Q$ are always values of the base field~$K$, whereas we need them to be elements of the residue field~$k(Q)$. Tate's residue is related to ours via the trace map~$\Tr_{k(Q)/K} \colon k(Q) \to K$. 
\end{rem}

We first state the residue theorem in the case where all cusps are $K$-rational, the general version is given in \Cref{thm:residue-theorem-general} below. 

\begin{thm}[$p$-adic residue theorem, case of rational cusps]
	\label{thm:residue-theorem}
	Let $f \in k(X)^\times$ be a rational function whose divisor has support in~$Y$ and let~$\omega$ be a log differential on~$(X,D)$. Assume that all points in~$D$ are $K$-rational. Then we have
	\begin{equation}
		\label{eq:residue-theorem}
		\int_{\div(f)} \omega = \sum_{Q \in D(K)} \Res_Q(\omega) \log f(Q).
	\end{equation}
\end{thm}

\begin{proof}
	As explained in \cite[Remarque~II.2.10]{colmez}, this formula can be derived from the reciprocity law for differentials of the third kind \cite[Théorème~6]{colmez} applied to $\alpha_1 = \omega$ and $\alpha_2 = \frac{\rd f}{f}$. We sketch the argument here. By linearity, we can assume that $\omega$ has integer residues, so that it is a differential of the third kind in the sense of \cite{colmez}. The reciprocity law says that the difference of the two integrals in the equation
	\begin{equation}
		\label{eq:reciprocity-law}
		\int_{\Res\left(\frac{\rd f}{f}\right)} \omega = \int_{\Res(\omega)} \frac{\rd f}{f}
	\end{equation}
	is given by a certain cup product which vanishes when one of the differentials is of the form~$\frac{\rd f}{f}$. Here, $\Res(\omega) = \sum_Q \Res_Q(\omega) [Q]$ denotes the residue divisor of a differential form. Using the fact that $\Res(\rd f/f) = \div(f)$ and that the integral of $\frac{\rd f}{f}$ on a divisor~$\sum_Q m_Q[Q]$ is given by $\sum_Q m_Q \log f(Q)$, Equation~\eqref{eq:reciprocity-law} becomes~\eqref{eq:residue-theorem}. We remark that Colmez works under the assumption that the genus of~$X$ is at least~one but this assumption does not appear to be used. 
\end{proof}

\begin{rem}
	\label{rem:residue-theorem-direct-proof}
	\Cref{thm:residue-theorem} can also be proved directly from the definition of the abelian integral, without reference to the reciprocity law for differentials of the third kind. This proof boils down to showing that the bottom map in the commutative diagram
	\[ \begin{tikzcd}
			T_D(K) \dar["\log_{T_D}"] \rar & J_Y(K) \dar["\log_{J_Y}"] \\
			\Lie(T_D) \rar["\phi"] & \Lie(J_Y)
	\end{tikzcd} \]
	can be identified with
	\begin{align}
		\label{eq:residue-theorem-phi-map}
		\phi\colon \Bigl(\prod_{Q \in D(K)} K\Bigr)\bigm/K \to \rH^0(X, \Omega^1(D))^\vee, \quad
		\phi((b_Q)_Q)(\omega) \coloneqq \sum_Q \Res_Q(\omega) b_Q,
	\end{align}
	where $T_D = \bigl(\prod_{Q \in D(K)} \bG_m\bigr)/\bG_m$ denotes the toric part of the generalised Jacobian. Indeed, the principal divisor $\div(f) \in J_Y(K)$ is the image of the tuple $(f(Q))_{Q \in D(K)} \in T_D(K)$ along the top row of the diagram, and spelling out the commutativity yields precisely~\eqref{eq:residue-theorem}. We omit the details of how to identify the map~$\phi$ in the diagram with~\eqref{eq:residue-theorem-phi-map}.
\end{rem}

We can remove the restriction that the cusps be $K$-rational at the expense of adding the trace maps $\Tr_{k(Q)/K}\colon k(Q) \to K$ of the residue field extensions at the cusps in the formula:

\begin{thm}[$p$-adic residue theorem, general version]
	\label{thm:residue-theorem-general}
	Let $f \in k(X)^\times$ be a rational function whose divisor has support in~$Y$ and let~$\omega$ be a log differential on~$(X,D)$. Then we have
	\begin{equation}
		\label{eq:residue-theorem-general}
		\int_{\div(f)} \omega = \sum_{Q \in |D|} \Tr_{k(Q)/K}\bigl(\Res_Q(\omega) \log f(Q)\bigr).
	\end{equation}
\end{thm}

\begin{proof}
	Let $L$ be a finite extension of~$K$ in~$\bC_p$ such that all points in~$D$ are defined over~$L$. Then the right hand side of~\eqref{eq:residue-theorem} can be written as $\sum_{Q'} \Res_{Q'}(\omega) \log f(Q')$, where $Q'$ runs through $D(L)$ and the residues are now taken on the base changed curve~$X_L$. Here we use the fact that a uniformiser at~$Q$ on~$X$ remains a uniformiser at each point~$Q'$ over~$Q$ on~$X_L$, so that $\Res_Q(\omega) = \Res_{Q'}(\omega)$ holds in~$L$. The integral on the left hand side of \eqref{eq:residue-theorem} is also invariant under extending the base field, so we may apply \Cref{thm:residue-theorem}. %
\end{proof}

\section{Annihilating log differentials}
\label{sec:explicit}

As in \Cref{sec:affchab1}, let $X$ be a smooth proper curve over a number field ~$K$, let $D \neq \emptyset$ be a finite set of closed points and set $Y = X \smallsetminus D$.
Assume $D$ does not consist of a single $K$-rational point.
Let $\cX$, $\cD$ and $\cY$ be as before.
Fix a base point $P_0 \in Y(K)$, an $s$-dimensional linear subspace $U$ of the $\bQ$-vector space $V_D$ defined in \eqref{eq:global-sigma} and an element $b \in V_D$. %
Moreover, fix an auxiliary prime~$\fp \not\in S$.
Let $W \coloneqq \sigma^{-1}(U)$, which is a linear subspace of $J_Y(K)_{\bQ}$, and let $b_0\in J_Y(K)_{\bQ}$ with $\sigma(b_0) = b$. %
If
\begin{equation}
\label{eq:chabauty-condition-affsubspace}
r + s + ([K : \bQ]-1)n < g + \rank \cO_K^\times + \#|D| + n_2(D) - 1,
\end{equation}
then the image of $W$ under $\log_{J_Y}$ is contained in a proper $K_{\fp}$-linear subspace of $\rH^0(X_{K_{\fp}},\Omega^1(D))^\vee$.
Hence there is a non-zero log differential $\omega=\omega(U)$ such that every point $P\in Y(K)$ with $\AJ_{P_0}(P)\in W$ satisfies %
\[
	\int_{P_0}^P \omega = 0.
\]
Moreover, there is a constant $c=c(P_0,U,b)\in K_{\fp}$ such that every point $P\in Y(K)$ with $\AJ_{P_0}(P)\in b_0 + W$ satisfies
$
	\int_{P_0}^P \omega = c.
$
In this section, we explain how $\omega$ and $c$ can be computed in practice.
In later applications, we will partition the $S$-integral points according to their reduction types $\Sigma=(\Sigma_{\fq})_{\fq}$ and set $b_0 + W$ equal to the Selmer subset $\Sel_{\bQ}(P_0,\Sigma)$, i.e. take $b + U=\fS_{\bQ}(P_0,\Sigma)$.
After some preliminary computations, we determine the differential $\omega$ by solving a system of linear equations. %
As outlined in \Cref{sec:algorithm}, the required inputs and preliminary computations are the following:
\begin{enumerate}[label=(\arabic*)] %
	
	\item \label{item:divisors}
	Divisors $G_1,\ldots,G_r \in \Div^0(Y)$ generating a full rank subgroup of the Mordell--Weil group $J(K)$. Finding these is a difficult computational problem in general, but for curves of genus at most two this can often be achieved by a 2-descent calculation as described in \cite{stoll:two-descent} and implemented in Magma via functions like \texttt{MordellWeilGroup}. Mordell--Weil generators for elliptic curves and genus-2 curves are also listed in the LMFDB \cite{lmfdb}. 
	By adding a principal divisor or taking a multiple if necessary, we can ensure that the $G_i$ have support in $Y$.
	
	\item \label{item:psi}
	For each $G_i$ chosen in~\ref{item:divisors} and each prime~$\fq$ of $K$ a vertical $\bQ$-divisor $\Phi_{\fq}(G_i)$ on~$\cX$ in the fibre over~$\fq$ such that $\Psi_{\fq}(G_i) \coloneqq \cG_i + \Phi_{\fq}(G_i)$ has $\fq$-intersection number zero with every vertical divisor on~$\cX$ over~$\fq$. Here, $\cG_i$ denotes the extension of~$G_i$ to a horizontal divisor on~$\cX$ obtained by taking Zariski closures of points. The $\Phi_{\fq}(G_i)$ can be computed in terms of the intersection matrix of the mod-$\fq$ fibre of~$\cX$, see \cite[Lemma~2.7]{affchab1}.
	When $\fq$ is a prime of good reduction for~$\cX$ we take $\Phi_{\fq}(G_i) = 0$ for all~$i$, which leaves only finitely many $\Phi_{\fq}(G_i)$ to be determined. 
	
	\item \label{item:differentials}
	A basis $\omega_1,\ldots,\omega_{g+n-1}$ of the space of log differentials on $(X,D)$.
	We also need the residues $\Res_Q(\omega_j)$ at the cusps $Q \in |D|$, which can be computed using \Cref{def:residue}.
	We will always assume that $\omega_1,\ldots,\omega_g$ form a basis of the space of regular differentials; this assumption simplifies the calculations as then $\Res_Q(\omega_j) = 0$ for all $Q\in|D|$ and $1\leq j\leq g$.
	
	\item \label{item:units}
	Elements $e_i = (e_{i,Q})_Q \in \prod_{Q\in|D|} \cO_{k(Q)}^\times$ for $1\leq i \leq k \coloneqq n_1(D)+n_2(D)-\# D - \rank \cO_K^\times$ generating a full rank subgroup of $\left(\prod_{Q\in|D|} \cO_{k(Q)}^\times\right)/\cO_K^\times$.

	\item \label{item:affspace}
	A basis $u_1,\dots,u_s$ of the subspace $U$ of $V_D$. The $u_i = \sum_{Q\in |D|} \sum_{\lambda} u_{i,\lambda} [\lambda]$ as well as the fixed vector $b = \sum_{Q\in |D|} \sum_{\lambda} b_\lambda [\lambda]$ are given as elements of $Z_0(\tilde{\cD})_\bQ$, with $u_{i,\lambda},b_\lambda \in \bQ$.
	Here $\lambda$ runs through the primes of $k(Q)$ and $[\lambda]$ denotes the prime~$\lambda$ viewed as a 0-cycle on $\tilde \cQ \subseteq \tilde \cD$, where $\tilde\cQ$ denotes the normalisation of the closure $\cQ$ of $Q$ in $\cX$.	
	We let $W \coloneqq \sigma^{-1}(U)$, which is a subspace of $J_Y(K)_{\bQ}$.
	
	\item \label{item:piK}
	For each prime $\fq$ of $\cO_K$ a generator $\rho_{\fq}$ of $\fq$ in $K^\times \otimes_{\bZ} \bQ$.
	Here we use that every prime ideal becomes principal after tensoring the group of fractional ideals with~$\bQ$.
	Concretely, we can take $m \geq 1$ such that $\fq^m$ becomes principal in $\cO_{K}$, say $\fq^m = (a)$, and then set $\rho_{\fq} = a^{1/m} \coloneqq a \otimes \frac1{m} \in K^\times \otimes \bQ$.
	
	\item \label{item:pi}
	For each cusp $Q \in |D|$ and each prime~$\lambda$ of $\cO_{k(Q)}$, a generator~$\pi_{\lambda}$ of $\lambda$ in $k(Q)^\times \otimes_{\bZ} \bQ$.
	We moreover choose the $\pi_{\lambda}$ in such a way that for each prime~$\fq$ of $\cO_K$ and each cusp $Q$ we have
	\begin{equation}
		\label{eq:pi-lambda-assumption}
		\prod_{\lambda|\fq} \pi_{\lambda}^{e(\lambda|\fq)} = \rho_{\fq} \quad \text{in } k(Q)^\times \otimes \bQ,
	\end{equation}
	where the product runs over all primes $\lambda$ of $\cO_{k(Q)}$ lying over~$\fq$ and $e(\lambda|\fq)$ denotes the ramification index. The equality~\eqref{eq:pi-lambda-assumption} always holds up to a unit in $\cO_{k(Q)}^\times \otimes \bQ$, and we can make sure that it holds exactly by absorbing the unit into one of the $\pi_{\lambda}$.
	In the formulas that follow, only finitely many of the $\pi_{\lambda}$ matter, namely those where the underlying prime $\fq$ is of bad reduction for~$\cX$, or the horizontal divisors~$\cG_i$ and $\cD$ intersect in the mod-$\fq$ fibre, or $u_{i,\lambda}\neq 0$, or the base point $P_0$ is not integral at~$\fq$. %
\end{enumerate}

Given these data we construct an $(r + n_1(D) + n_2(D) - \#|D| + s) \times (g+n-1)$ matrix
\begin{equation}
	\label{eq:MSigma}
	M(U) \coloneqq \begin{pmatrix}
		A & B  \\
		0 & C  \\
		0 & D(U) 
	\end{pmatrix}
\end{equation}
with $K_{\fp}$-entries by defining the submatrices $A$, $B$, $C$, $D(U)$ as follows:
\begin{enumerate} [label=(\Alph*)]
	\item $A$ is the $r\times g$-matrix with entries
	\begin{equation}
		\label{eq:matrixA}
		A_{i,j} \coloneqq \int_{G_i} \omega_j, \quad 1\leq i\leq r, 1\leq j \leq g
	\end{equation}

	\item $B$ is the $r \times (n-1)$-matrix with entries $B_{i,j} \coloneqq H(G_i,\omega_{g+j})$ for $1\leq i\leq r$, $1\leq j \leq n-1$, 
	where the pairing
	$ H\colon J_Y(K) \times \rH^0(X, \Omega^1(D)) \to K_{\fp} $
	is defined by $H(G, \omega) \coloneqq$
	\begin{equation}
		\label{eq:H}
		 \int_G \omega - \sum_{\substack{Q \in |D|,\\ \varphi\colon k(Q) \to \Kpbar}} \varphi(\Res_Q \omega) \sum_{\substack{\text{$\fq$ prime,}\\ \text{$\lambda|\fq$ in $\cO_{k(Q)}$}}} i_{\lambda}(\Psi_{\fq}(G), \tilde \cQ) \log \varphi(\pi_{\lambda}).
	\end{equation}
	Here, $Q$ runs over all cusps and $\varphi$ runs over all embeddings of the residue field~$k(Q)$ into~$\Kpbar$. In the inner sum, $\fq$ runs over all  primes of $K$, $\lambda$ runs over all primes of $\cO_{k(Q)}$ dividing~$\fq$ and $i_{\lambda}(\Psi_{\fq}(G), \tilde\cQ)$ refers to the intersection number defined in \S\ref{sec:intersection} when $\lambda$ is viewed as a closed point of $\tilde\cQ = \Spec(\cO_{k(Q)})$.
	The logarithm in \eqref{eq:H} is the $p$-adic logarithm on $\Kpbar^\times$, which naturally extends to $\Kpbar^\times \otimes \bQ$ via $\log(x^{1/m}) \coloneqq \frac1{m} \log(x)$. %
	The intersection number is zero except possibly for those primes~$\fq$ where $\cX$ has bad reduction or the horizontal divisors~$\cG$ and~$\cD$ intersect in the fibre over~$\fq$. In particular, the inner sum in \eqref{eq:H} is finite. Note that $H(G, \omega)$ is an element of $K_{\fp} \subseteq \Kpbar$ by Galois invariance. The divisor $\Psi_{\fq}(G)$ is only well-defined up to adding a $\bQ$-multiple of the complete mod-$\fq$ fibre $\cX_{\fq}$ but $H(G,\omega)$ is still well-defined since for fixed~$\fq$ we have
	\begin{align*}
		&\qquad\sum_{\substack{Q \in |D|,\\ \varphi\colon k(Q) \to \Kpbar}} \varphi(\Res_Q \omega) \sum_{\text{$\lambda|\fq$ in $\cO_{k(Q)}$}} i_{\lambda}(\cX_{\fq}, \tilde \cQ) \log \varphi(\pi_{\lambda}) \\
		&\qquad = \sum_{Q,\varphi} \varphi(\Res_Q \omega) \sum_{\lambda|\fq} e(\lambda|\fq) \log \varphi(\pi_{\lambda})
		\overset{\eqref{eq:pi-lambda-assumption}}{=} \log(\rho_{\fq}) \sum_{\overline{Q} \in D(\Kpbar)} \Res_{\overline{Q}}(\omega) = 0
	\end{align*}
	by the residue formula \cite[Proposition~II.6]{serrealggp}. The fact that $H(G,\omega)$ is well-defined on the generalised Jacobian, i.e.\ $H(\div(f), \omega) = 0$ for rational functions~$f \in k(X)^\times$ with $f\vert_D = 1$, follows from the fact that $i_{\lambda}(\div(f), \tilde \cQ) = v_{\lambda}(f(Q))$ by definition of the intersection number, where $v_{\lambda}$ denotes the normalised valuation at $\lambda$. Unless all cusps are rational or imaginary quadratic, the pairing $H(-,-)$ depends on the choice of the prime ideal generators~$\pi_{\lambda}$.
	
	\item $C$ is the $(n_1(D)+n_2(D)-\#|D|-\rank\cO_K^\times) \times (n-1)$-matrix with entries
	\begin{equation}
		\label{eq:matrixC}
		C_{i,j} \coloneqq 
		\sum_{\substack{Q \in |D|,\\ \varphi\colon k(Q) \to \Kpbar}} \varphi(\Res_Q \omega_{g+j}) \log \varphi(e_{i,Q}). %
	\end{equation}
	The matrix has one row for each element~$e_i$ that was chosen in step~\ref{item:units}. If all cusps are rational or imaginary quadratic, the matrix~$C$ is empty.
	
	\item $D(U)$ is the $s \times (n-1)$-matrix with entries
	\begin{equation}
		\label{eq:matrixD}
		D(U)_{i,j} \coloneqq \sum_{\substack{Q \in |D|,\\ \varphi\colon k(Q) \to \Kpbar}} \varphi(\Res_{Q} \omega_{g+j}) \sum_{\substack{\text{$\fq$ prime,}\\ \text{$\lambda|\fq$ in $\cO_{k(Q)}$}}} u_{i,\lambda} \log \varphi(\pi_{\lambda}).
	\end{equation}
	Here $u_{i,\lambda}$ are as in step \ref{item:affspace} above.
	
\end{enumerate}

\begin{rem}
	\label{rem:global-height}
	When $K$ is $\bQ$ or imaginary quadratic and the residues of $\omega_j$ are in $K$, the expressions for the entries of $M(U)$ simplify.
	Namely, the $j$-th column of $C$ is $0$, and the entries in the $j$-th column of $B$ (and the constants $c(b,\omega)$ in \eqref{eq:constants} below) can be expressed in terms of $\fq$-intersection numbers with the residue divisor $R = \sum_{Q \in |D|} \Res_{Q}(\omega_j) [Q]$.
	If moreover $K=\bQ$, then $H(G,\omega_j)$ agrees with $h(G,R)$, the global Coleman--Gross height pairing \cite{CG89}, for a suitable choice of splitting of the Hodge filtration on $\rH^1_{\dR}(X/\bQ_p)$.
\end{rem}

When Condition~\eqref{eq:chabauty-condition-affsubspace} is satisfied, the number of rows of~$M(U)$ is strictly smaller than the number of columns, so the matrix has a non-trivial kernel.

\begin{thm}
	\label{thm:explicit-function}
	If $(a_1,\ldots,a_{g+n-1}) \in K_{\fp}^{g+n-1}$ is contained in the kernel of the matrix $M(U)$ defined in equations~\eqref{eq:MSigma}--\eqref{eq:matrixD}, let $\omega = \omega(U) \coloneqq a_1 \omega_1 + \dots + a_{g+n-1} \omega_{g+n-1}$.
	Define
	\begin{equation}
	\label{eq:constants}
	c = c(b, \omega) \coloneqq \sum_{\substack{Q \in |D|,\\ \varphi\colon k(Q) \to \Kpbar}} \varphi(\Res_Q \omega) \sum_{\textnormal{$\lambda$ prime of $\cO_{k(Q)}$}} b_{\lambda} \log \varphi(\pi_{\lambda}).
	\end{equation}
	Then
	\[
		\int_F \omega = c \quad \text{ for all } F \in b_0 + W.
	\]
\end{thm}

The constant in \eqref{eq:constants} can be expressed in terms of $b$, the $a_j$, and the residues of the basis differentials $\omega_j$.

\begin{proof}
By \cite[Proposition 3.7]{affchab1} and the surjectivity of $\sigma\colon J_Y(K)_{\bQ} \to V_D$, we see that $W$ is a subspace of $J_Y(K)_{\bQ}$ of dimension
	\[ \dim_{\bQ} W = n_1(D) + n_2(D) - \#|D| - \rank \cO_K^\times + r + s. \]
	
Similar to~\eqref{eq:chabauty-diagram}, we look at the composite %
	\begin{equation}
		\label{eq:rational-chabauty-diagram}
		\begin{tikzcd}
			W \rar[hook]  & J_Y(K)_{\bQ} \rar[hook] & J_Y(K_{\fp})_{\bQ} \rar["\log_{J_Y}"] & \rH^0(X_{K_{\fp}}, \Omega^1(D))^\vee. %
		\end{tikzcd}
	\end{equation}
	Dually, we get the $K_{\fp}$-linear map
	\begin{equation}
	\label{eq:ck-pullback-map}
		\log_{J_Y}^\sharp \colon \rH^0(X_{K_{\fp}}, \Omega^1(D)) \to \Hom(W, K_{\fp}) = W^\vee\otimes_{\bQ} K_{\fp}.
	\end{equation}
	We show that the matrix $M(U)$ represents this map for a suitable choice of bases. A basis of $\rH^0(X_{K_{\fp}}, \Omega^1(D))$ is given by $\cB = (\omega_1,\ldots,\omega_{g+n-1})$. We construct a basis of $W^\vee$ as follows.
	
	Since the classes of the divisors $G_1,\ldots,G_r$ chosen in step~\ref{item:divisors} form a basis of $J(K)_{\bQ}$, any $F \in \Div^0(Y)_{\bQ}$ can be written as
	\begin{equation}
		\label{eq:F-and-Gi}
		F = \sum_{i=1}^r x_i(F) G_i + \div(f)
	\end{equation}
	with uniquely determined $x_i(F) \in \bQ$ and with $f \in k(X)^\times \otimes \bQ$ determined up to scalars. The $x_i(F)$ depend only on the class of~$F$ in $J(K)_{\bQ}$; in particular, the $x_i$ are well-defined functions on $J_Y(K)_{\bQ}$ and on the subspace $W$. 
	
	Since $F$ and the $G_i$ have support in~$Y$, the function $f \in k(X)^\times_{\bQ}$ has no zeros nor poles at the cusps, so we can evaluate it at each $Q \in |D|$ to obtain $f(Q) \in k(Q)^\times_{\bQ} \coloneqq k(Q)^\times \otimes \bQ$. Using the basis $e_1,\dots,e_k$ of $\Bigl(\prod_Q \cO_{k(Q)}^\times\Bigr)/\cO_K^\times \otimes \bQ$ chosen in step~\ref{item:units} and the prime ideal generators $\pi_{\lambda}$ chosen in step~\ref{item:pi}, the values of $f$ at the cusps factor as
	\begin{equation}
		\label{eq:f-of-Q}
		(f(Q))_Q = \prod_{i=1}^{k} e_{i}^{a_{i}(F)} \cdot  \left(\prod_{\text{$\lambda$ prime of $\cO_{k(Q)}$}} \pi_{\lambda}^{v_\lambda(f(Q))}\right)_Q \quad \text{in} \quad \left.\left(\prod_Q k(Q)^\times_{\bQ}\right)\right/K_{\bQ}^\times,
	\end{equation}
	with $v_{\lambda} \colon k(Q)^\times_{\bQ} \to \bQ$ denoting the $\lambda$-adic valuation. The exponents $a_{i}(F) \in \bQ$ do not depend on the choice of~$f$.
	Indeed, rescaling $f$ by $\rho_{\fq}$ changes only the exponents of the $\pi_{\lambda}$, thanks to the factorisation \eqref{eq:pi-lambda-assumption}; and rescaling $f$ by a unit in $\cO_K^\times$ has no effect on either side of \eqref{eq:f-of-Q}.
	Moreover, the $a_{i}(F)$ depend only on the $D$-equivalence class of~$F$, since multiplying~$f$ by a function $g$ with $g\vert_D = 1$ does not change the values $f(Q)$. So the $a_{i}$ are well-defined functions on $J_Y(K)_{\bQ}$ and by restriction also on $W$.
	Once we choose a particular $f$, \eqref{eq:f-of-Q} says there exists a constant $\kappa \in K_{\bQ}^\times$ such that for each $Q\in |D|$ we have
	\begin{equation}
	\label{eq:f-of-Q2}
	f(Q)= \left(\prod_{i=1}^{k} e_{i,Q}^{a_{i}(F)}\right) \cdot  \left(\prod_{\text{$\lambda$ prime of $\cO_{k(Q)}$}} \pi_{\lambda}^{v_\lambda(f(Q))}\right) \cdot \kappa \quad \text{in} \quad k(Q)^\times_{\bQ}.
	\end{equation}
	
	For $F \in W$, we have $\sigma(F) \in U = \langle u_1,\ldots,u_s\rangle$ by definition, so we can write
	\begin{equation}
		\label{eq:t-ell}
		\sigma(F) = \sum_{i=1}^s t_{i}(F) u_i \qquad \text{in $V_D$},
	\end{equation}
	with uniquely determined $t_{i}(F) \in \bQ$. Now
	$\cC \coloneqq (x_1,\ldots,x_r, a_1,\ldots,a_k, t_1,\ldots, t_s)$
	is the desired basis of $W^\vee$.
	
	We now show that the matrix $M(U)$ represents the map $\log_{J_Y}^{\sharp}$ of~\eqref{eq:ck-pullback-map} in the bases~$\cB$ and~$\cC$. %
	Let $\omega \in \rH^0(X_{K_{\fp}}, \Omega^1(D))$ be a log differential and let $F$ be a $\bQ$-divisor on~$Y$ of degree zero representing an element of $W$. We want to compute $\log_{J_Y}^{\sharp}(\omega)(F) = \int_F \omega$. With~\eqref{eq:F-and-Gi} we have
	\begin{equation}
		\label{eq:int-F-eta}
		\int_F \omega = \sum_{i=1}^r x_i(F) \int_{G_i} \omega + \int_{\div(f)} \omega.
	\end{equation}
	When $\omega = \omega_j$ ($j=1,\ldots,g$) is a holomorphic differential, the integral over the principal divisor vanishes, %
	explaining the block matrix~$A$ and the zeros underneath it in the first~$g$ columns of~$M(U)$. In general, we use the $p$-adic residue theorem (\Cref{thm:residue-theorem}) to rewrite the integral over $\div(f)$ in~\eqref{eq:int-F-eta}:
	\begin{align}
		\int_{\div(f)} \omega = \sum_{\substack{Q \in |D|,\\ \varphi\colon k(Q) \to \Kpbar}} \varphi(\Res_Q \omega) \log \varphi(f(Q)) \label{eq:int-princ-div}. %
	\end{align}
	By \eqref{eq:f-of-Q2} we have
	\begin{equation}
		\label{eq:log-f-Q}
		\log \varphi(f(Q)) = \sum_{i=1}^{k} a_{i}(F) \log \varphi(e_{i,Q}) + \sum_{\text{$\lambda$ prime of $\cO_{k(Q)}$}} v_{\lambda}(f(Q)) \log \varphi(\pi_{\lambda}) + \log(\kappa).
	\end{equation}
	The valuations~$v_\lambda(f(Q))$ are encoded in $\sigma(\div(f))$ since we have
	\begin{equation}
		\label{eq:sigma-div-f}
		\sigma(\div(f)) = \sum_{Q \in |D|}\; \sum_{\text{$\lambda$ prime of $\cO_{k(Q)}$}} v_{\lambda}(f(Q)) [\lambda]
	\end{equation}
	by \cite[Proposition 3.4]{affchab1}. From~\eqref{eq:F-and-Gi} and~\eqref{eq:t-ell} we have
	\begin{align}
		\sigma(\div(f)) = \sigma(F) - \sum_{i=1}^r x_i(F) \sigma(G_i) 
		 = \sum_{i=1}^s t_{i}(F) u_i  - \sum_{i=1}^r x_i(F) \Psi(G_i).\tilde\cD. \label{eq:sigma-div-f-with-Gi}
	\end{align}
	This is an equality in $V_D$.
	Combining \eqref{eq:sigma-div-f} and \eqref{eq:sigma-div-f-with-Gi}, we conclude that there is a $z\in \bQ$ such that for all prime ideals $\fq$ of $K$, all $Q\in |D|$ and all $\lambda|\fq$ in $k(Q)$ we have %
	\begin{equation}
		\label{eq:v-lambda}
		v_\lambda(f(Q)) = \sum_{i=1}^s t_{i}(F) u_{i,\lambda}  - \sum_{i=1}^r x_i(F) i_\lambda(\Psi_{\fq}(G_i),\tilde{\cD}) + z e(\lambda|\fq).
	\end{equation}
	Putting everything together we obtain
	\begin{align*}
		\int_F \omega \overset{\eqref{eq:int-F-eta},\eqref{eq:int-princ-div}}&{=} \sum_{i=1}^r x_i(F) \int_{G_i} \omega + \sum_{\substack{Q \in |D|,\\ \varphi\colon k(Q) \to \Kpbar}} \varphi(\Res_Q \omega) \log \varphi(f(Q)) \\
		\overset{\eqref{eq:log-f-Q},\eqref{eq:v-lambda}}&{=} \sum_{i=1}^r x_i(F) \int_{G_i} \omega + \sum_{Q,\varphi} \varphi(\Res_Q \omega) \sum_{i=1}^{k} a_{i}(F) \log \varphi(e_{i,Q}) \\
		&\qquad - \sum_{Q,\varphi} \varphi(\Res_Q \omega) \sum_{\substack{\text{$\fq$ prime,}\\ \text{$\lambda|\fq$ in $\cO_{k(Q)}$}}} \Bigl(\sum_{i=1}^r x_i(F) i_{\lambda}(\Psi_{\fq}(G_i), \tilde \cQ) \Bigr) \log \varphi(\pi_{\lambda}) \\
		&\qquad + \sum_{Q,\varphi} \varphi(\Res_{Q} \omega) \sum_{i=1}^s \; t_{i}(F) \sum_{\substack{\text{$\fq$ prime,}\\ \text{$\lambda|\fq$ in $\cO_{k(Q)}$}}} u_{i,\lambda} \log \varphi(\pi_{\lambda_i}) \\
		&\qquad + \log(\kappa) \sum_{Q,\varphi} \varphi(\Res_{Q} \omega) + \sum_{Q,\varphi} \varphi(\Res_{Q} \omega) \sum_{\fq, \lambda|\fq} z e(\lambda|\fq) \log \varphi(\pi_\lambda).
	\end{align*}
Note that the summands in the final line vanish by assumption \eqref{eq:pi-lambda-assumption} and the residue formula $\sum_{Q,\varphi} \varphi(\Res_Q \omega) = 0$. Rearranging gives
\begin{align*}
		\int_F \omega &= \sum_{i=1}^r H(G_i, \omega) x_i(F) 
		+ \sum_{i=1}^{k} \left( \sum_{\substack{Q \in |D|,\\ \varphi\colon k(Q) \to \Kpbar}} \varphi(\Res_Q \omega) \log \varphi(e_{i,Q}) \right) a_{i}(F)  \\
		&\qquad + \sum_{i=1}^s \left( \sum_{\substack{Q \in |D|,\\ \varphi\colon k(Q) \to \Kpbar}} \varphi(\Res_{Q} \omega) \sum_{\substack{\text{$\fq$ prime,}\\ \text{$\lambda|\fq$ in $\cO_{k(Q)}$}}} u_{i,\lambda} \log \varphi(\pi_{\lambda}) \right) t_{i}(F). %
	\end{align*}
	This is indeed a linear combination of $x_i(F)$, $a_{i}(F)$, and $t_{i}(F)$. When we take~$\omega$ to be one of the basis differentials $\omega_{g+j}$ ($j=1,\ldots,n-1$), the coefficients form the $(g+j)$-th column of the matrix~$M(U)$, explaining the block matrices $B$, $C$, and $D(U)$ in the definition of~$M(U)$ and concluding the proof that $M(U)$ represents the map~\eqref{eq:ck-pullback-map} in the bases~$\cB$ and~$\cC$.
	
	If $(a_1,\ldots,a_{g+n-1})$ is in the kernel of $M(U)$ then $\omega \coloneqq \sum_{j=1}^{g+n-1} a_j \omega_j$ pulls back to zero under $\log_{J_Y}$, so that $\int_F \omega = 0$ for $F\in W$.
	If instead $F \in b_0 + W$, i.e. $\sigma(F) \in b + U$, then $\int_F \omega$ takes a constant value independent of $F$, which we compute as follows:
	We choose a function $\beta\in  k(X)^\times_{\bQ}$ satisfying $\beta(Q) = \prod_{\text{$\lambda$ prime of $\cO_{k(Q)}$}} \pi_\lambda^{b_\lambda}$ for all $Q$.
	Then $\sigma(\div(\beta)) = b$, so that $b_0 + W = \div(\beta) + W$ and hence, by \Cref{thm:residue-theorem-general}, every $F\in b_0 + W$ satisfies
	\[
		\int_F \omega = \int_{\div(\beta)} \omega = \sum_{\substack{Q \in |D|,\\ \varphi\colon k(Q) \to \Kpbar}} \varphi(\Res_Q \omega) \log \varphi\left(\textstyle\prod_\lambda \pi_\lambda^{b_\lambda}\right)  = c(b,\omega). \qedhere
	\]
\end{proof}

In order to show \Cref{thm:A}, let $\Sigma$ be an $S$-integral reduction type with cuspidal part $\Sigma^{\csp} = (\Sigma_{\fq})_{\fq\in S_0}$. We apply \Cref{thm:explicit-function} with $b+U \coloneqq \fS_{\bQ}(P_0,\Sigma)$ as defined in \Cref{sec:selmerset}.
That is, $b$ is defined as in \eqref{eq:b-lambda} and 
a basis of $U = U_{\bQ}(\Sigma^{\csp})$ is given by the elements $u_{\fq} = [\Sigma_{\fq}]$ for $\fq\in S_0$.
We will write $M(\Sigma^{\csp})$ for $M(U)$, $\omega=\omega(\Sigma^{\csp})$ for any $\omega(U)$ coming from \Cref{thm:explicit-function} and $c(P_0,\Sigma,\omega)$ for $c(b,\omega)$. 
The existence of a nonzero $\omega$ is guaranteed if \eqref{eq:chabauty-condition} holds.

\begin{proof}[Proof of \Cref{thm:A}]
	By \cite[Proposition 3.13]{affchab1}, $\AJ_{P_0}(\cY(\cO_{K,S})_\Sigma)$ is contained in $\Sel_{\bQ}(P_0,\Sigma) = \sigma^{-1}(\fS_{\bQ}(P_0,\Sigma))$, thus the claim follows from \Cref{thm:explicit-function}.
\end{proof}

When one knows enough $S$-integral points with the same cuspidal reduction type, one can also derive a Chabauty function in the form of a determinant. %

\begin{thm}
	\label{thm:determinant-equation-fixed-cuspidal-part}
	Let $S_0 \subseteq S$ and fix $\Sigma^{\csp} = (\Sigma_{\fq})_{\fq \in S_0}$ with $\Sigma_{\fq} \in (\cX_{\fq}^{\sm} \cap \cD)(\bF_{\fq})$.
	 Assume \eqref{eq:chabauty-condition} holds.
	 Then for any $g+n-1$ points $P_1,\ldots,P_{g+n-1} \in \cY(\cO_{K,S})$ whose reduction types $\Sigma_1,\ldots,\Sigma_{g+n-1}$ have cuspidal part~$\Sigma^{\csp}$, we have
	\begin{equation}
		\label{eq:determinant-equation-fixed-cuspidal-part}
		\det \begin{pmatrix}
		\int_{P_0}^{P_i} \omega_j - c(P_0, \Sigma_i, \omega_j)
	\end{pmatrix}_{1 \leq i,j \leq g+n-1} = 0.
	\end{equation}
\end{thm}

\begin{proof}
	The condition~\eqref{eq:chabauty-condition} implies that the matrix $M(\Sigma^{\csp})$ defined in \eqref{eq:MSigma} has nontrivial kernel. Let $0 \neq (a_1,\ldots,a_{g+n-1}) \in K_{\fp}^{g+n-1}$ be an element of the kernel and set $\omega = \sum_{j=1}^{g+n-1} a_j \omega_j$. By \Cref{thm:A}, we have
	$ 0 = \int_{P_0}^{P_i} \omega - c(P_0,\Sigma_i,\omega) = \sum_{j=1}^{g+n-1} a_j ( \int_{P_0}^{P_i} \omega_j - c(P_0, \Sigma_i, \omega_j)) $
	for all $i$, where we used that $c(P_0,\Sigma_i,\omega)$ is linear in~$\omega$. %
	Thus the columns of the matrix in~\eqref{eq:determinant-equation-fixed-cuspidal-part} are linearly dependent. %
\end{proof}

\section{Examples}
\label{sec:examples}

\subsection{Even hyperelliptic curves}
\label{sec:split-hyperelliptic}

Let $\cY/\bZ$ be an affine hyperelliptic curve defined by an equation $y^2 = f(x)$ with $f \in \bZ[x]$ squarefree of degree~$2g+2$ such that the leading coefficient of~$f$ is a square in~$\bZ$. Suppose that the Mordell--Weil rank~$r$ equals~$g$. Let~$p$ be a prime of good reduction. In \cite{GM:LinearQuadraticChabauty}, the authors use $p$-adic heights to construct a nontrivial function $\rho\colon \cY(\bZ_p) \to \bQ_p$ and a finite set~$T$ such that $\rho(\cY(\bZ)) \subseteq T$. We explain how such functions can be obtained with our methods, without the need to compute $p$-adic heights above~$p$ for some choice of complementary subspace of the holomorphic differentials in $\rH^1_{\dR}(X/ \bQ_p)$.

Let $Y/\bQ$ be the generic fibre of~$\cY$. The smooth compactification~$X$ of~$Y$ is obtained by taking the closure in the weighted projective space $\bP(1,g+1,1)$ over~$\bQ$. Denote the leading coefficient of~$f$ by~$d$. Since~$d$ is assumed to be a square, there are two $\bQ$-rational points at infinity:
\begin{align*}
	\infty_+ = (1 : \sqrt{d} : 0), \quad
	\infty_- = (1 : -\sqrt{d} : 0).
\end{align*}
Let $D = \{\infty_+, \infty_-\}$. Let $\cX$ be the model of~$X$ defined by the same equation in the weighted projective space $\bP(1,g+1,1)$ over~$\bZ$, and let $\cD$ be the closure of~$D$, so that $\cY = \cX \smallsetminus \cD$. 
We have $r = g$, $\#S = 0$, $\#|D| = 2$, $n_2(D) = 0$, so the Affine Chabauty Condition~\eqref{eq:chabauty-condition-over-number-field-intro} is satisfied.
Note that the model given by the equation $y^2 = f(x)$ might not be regular, so one first needs to find a desingularisation $\cX' \to \cX$, using for example the \texttt{RegularModel} function in Magma.
\cite[Theorem A]{affchab1} ensures the existence of a non-zero log differential $\omega \in \rH^0(X_{\bQ_p}, \Omega^1(D))$ such that the function $P \mapsto \int_{P_0}^P \omega$ is constant on $\cY(\bZ)_{\Sigma}$ for each integral reduction type~$\Sigma$. The differential~$\omega$ can be found by the methods of \Cref{sec:explicit}. A lot of the considerations simplify since both cusps are defined over~$\bQ$ and we are looking for integral points, so all reduction types have the same (empty) cuspidal part. If one knows sufficiently many integral points to start with, \Cref{thm:determinant-equation-fixed-cuspidal-part} yields a Chabauty function in the form of a determinant. In general, one needs to construct the matrix $M$ from \eqref{eq:MSigma} and compute its kernel. First we find degree zero divisors $G_1,\ldots,G_r$ supported on~$Y$ spanning a subgroup of finite index in the Jacobian $J(\bQ)$. A basis of $\rH^0(X, \Omega^1(D))$ is given by $\omega_0,\ldots,\omega_g$ where $\omega_j = x^j \tfrac{\rd x}{y}$ (note the index shift when compared to \ref{item:differentials}). The first~$g$ of these differentials are holomorphic and the last one has simple poles at the points at infinity with residues
\[ \Res_{\infty_+}(\omega_g) = -\tfrac1{\sqrt{d}}, \quad  \Res_{\infty_-}(\omega_g) = \tfrac1{\sqrt{d}} \]
by a similar computation as in \cite[Proposition~3.9]{GM:computing-p-adic-heights}. 
The $r \times (g+1)$-matrix~$M$ defined in \eqref{eq:MSigma} consists of the matrix $A$ defined in \eqref{eq:matrixA}, i.e.\ $M_{i,j} = \int_{G_i} \omega_{j-1}$ for $1\leq i \leq r$, $1 \leq j \leq g$, and a last column~$B$ whose entries are given by \eqref{eq:H}:
\begin{equation}
	\label{eq:hyp-M-entries}
	M_{i,g+1} = \int_{G_i} \omega_g + d^{-1/2}\sum_{\text{$\ell$ prime}} i_{\ell}(\Psi_{\ell}(G_i), \cQ'_+ - \cQ'_-) \log(\ell),
\end{equation}
where $\cQ_{\pm}'$ denotes the closure of~$\infty_{\pm}$ in the regular model~$\cX'$.
The Coleman integrals can be computed using the algorithm from \cite{balakrishnan:even-degree-coleman-integration} whose implementation in SageMath is available from \url{https://github.com/jbalakrishnan/AWS}. For this one should assume that the divisors~$G_i$ are supported in $\bQ_p$-rational points. The intersection numbers times $\log(\ell)$ in~\eqref{eq:hyp-M-entries} are minus the local heights $h_{\ell}(G_i, \infty_+ - \infty_-)$ away from~$p$, which appear also in \cite[§3.2]{GM:LinearQuadraticChabauty}.
We can then compute the kernel of the matrix~$M$. Since the Mordell--Weil rank~$r$ equals~$g$ by assumption, we find a nontrivial element $(a_0,\ldots,a_g) \in \bQ_p^{g+1}$ in the kernel. By \Cref{thm:A}, the log differential
\[ \omega \coloneqq a_0 \omega_0 + \dots + a_g \omega_g \]
has the property that the function $\rho\colon \cY(\bZ_p) \to \bQ_p$, $\rho(P) = \int_{P_0}^P \omega$ is constant on $\cY(\bZ)_{\Sigma}$ for every reduction type~$\Sigma$. %
The reduction type~$\Sigma = (\Sigma_{\ell})_{\ell}$ needs to be specified on the regular model~$\cX'$ (see \cite[\S 3.4]{affchab1}), i.e.\ $\Sigma_{\ell}$ is a multiplicity-one component of the mod-$\ell$ fibre $\cX'_{\ell}$.
The constant value~$c(P_0,\Sigma,\omega)$ of $\rho$ on $\cY(\bZ)_{\Sigma}$ is %
\begin{equation}
	\label{eq:split-hyperelliptic-constants-final}
	c(P_0,\Sigma,\omega) = \frac{a_g}{\sqrt{d}} \sum_{\ell \neq p} \Bigl(i_{\ell}(\cP_0', \cQ_+' - \cQ_-') - i_{\ell}(\Phi_{\ell}(\Sigma_{\ell} - \cpt_{\ell}(P_0)), \cQ_+' - \cQ_-') \Bigr) \log(\ell)
\end{equation}
by \eqref{eq:constants} and \eqref{eq:b-lambda}.
Here, the intersection numbers are computed on the regular model $\cX'$,  $\cP_0'$ denotes the closure of~$P_0$ in~$\cX'$, $\cpt_{\ell}(P_0)$ denotes the component of~$\cX'_{\ell}$ onto which the base point~$P_0$ reduces, and $\log(p)=0$ by our choice of branch of the $p$-adic logarithm.

\begin{rem}
	We can similarly calculate integral points on $\cY$ if the leading coefficient $d$ of $f$ is negative.
	In this case, there is one cusp $(1:\sqrt{d}:0)$ with imaginary quadratic residue field $\bQ(\sqrt{d})$, thus \eqref{eq:chabauty-condition-over-number-field-intro} holds if $r=g$.
\end{rem}

We found many examples where our method computes exactly the set of integral points. We illustrate this by revisiting \cite[Example 4.4]{GM:LinearQuadraticChabauty} and determining the $\bZ$-points of the genus-2 curve with LMFDB label \href{https://www.lmfdb.org/Genus2Curve/Q/6081/b/164187/1}{6081.b.164187.1} defined by
\[ \cY\colon\quad y^2 = x^6 + 2x^5 - 7x^4 - 18x^3 + 2x^2 + 20x + 9. \]

\begin{thm}
	\label{thm:example-split}
	The set $\cY(\bZ)$ is equal to
	\[ \left\{ (-1,\pm 1), \; (0, \pm 3), \; (1, \pm 3), \; (-2,\pm 3), \; (-4, \pm 37) \right\}. \]
\end{thm}

The model~$\cX$ given by the projective closure of~$Y$ in $\bP^2_{\bZ}$ has bad reduction at $2$, $3$, and $2027$. The model~$\cX_1/\bZ$ defined by the minimal equation
\[ y^2 + (x^3+x^2+1)y = -2x^4 - 5x^3 + 5x + 2 \]
has good reduction at~2 and the map $\cX_1 \to \cX$, $(x,y) \mapsto (x, 2y + x^3 + x^2 + 1)$ is an isomorphism over all odd primes, so the only bad primes for~$\cX_1$ are $3$ and $2027$. The \texttt{RegularModel} function in Magma shows that $\cX_1$ is regular in the fibre over~2027 but not over~3. The affine part of the mod-3 fibre of $\cX$ (and hence of~$\cX_1$) is given by the equation
\[ y^2 = x(x^3 + x^2 + 2)(x+2)^2 \]
which is irreducible and has an ordinary double point at $(\overline{1},\overline{0}) \in \cY(\bF_3)$.
Let $\cX' \to \cX_1 \to \cX$ be a desingularisation.

We choose the base point $P_0 = (-1,1) \in \cY(\bZ)$ and the auxiliary prime~$p = 7$.
Setting $P_1 = (0,3)$ and $P_2 = (1,3)$, the classes of the divisors $G_i = P_i - P_0$ ($i = 1,2$) generate a subgroup of finite index in the Mordell--Weil group. The intersection numbers $i_{\ell}(\Psi_{\ell}(G_i), \cQ_+' - \cQ_-')$ in~\eqref{eq:hyp-M-entries} are all zero: for the horizontal part $i_{\ell}(\cP_i - \cP_0, \cQ_+' - \cQ_-')$ this follows from $\cP_0,\cP_1,\cP_2$ being integral points. And all vertical divisors~$V$ in $\cX_{\ell}'$ satisfy $i_{\ell}(V, \cQ_+' - \cQ_-') = 0$ since the points $Q_{+}$ and $Q_-$ reduce onto the same component of~$\cX_{\ell}'$. For $\ell \neq 3$ this is clear since~$\cX'_{\ell}$ has only one component; for $\ell = 3$ this follows from the fact that the points at infinity reduce to smooth points of the mod-3 fibre of the non-regularised model~$\cX$, and since the desingularisation $\cX' \to \cX$ is an isomorphism over the non-singular locus, they still reduce onto the same component of~$\cX_3'$. We find that the matrix~$M$ is given by
\begin{equation*}
	M = \begin{pmatrix}
		\int_{P_0}^{P_1} \omega_0 & \int_{P_0}^{P_1} \omega_1 & \int_{P_0}^{P_1} \omega_2 \\[1mm]
		\int_{P_0}^{P_2} \omega_0 & \int_{P_0}^{P_2} \omega_1 & \int_{P_0}^{P_2} \omega_2
	\end{pmatrix},
\end{equation*}
without the correction terms in the last column. A nontrivial element of the kernel is $(a_0,a_1,a_2) \in \bQ_7^3$ with
\begin{align*}
	a_0 &= 1 + O(7^8), \\
	a_1 &= 5 + 3\cdot 7 + 3\cdot 7^2 + 5\cdot 7^3 + 3\cdot 7^4 + 2\cdot 7^6 + 2\cdot 7^7 + O(7^8), \\
	a_2 &= 5 + 6\cdot 7 + 6\cdot 7^2 + 7^3 + 4\cdot 7^4 + 6\cdot 7^5 + 5\cdot 7^6 + 3\cdot 7^7 + O(7^8).
\end{align*}
Setting $\omega \coloneqq a_0 \omega_0 + a_1 \omega_1 + a_2 \omega_2$, the function $\rho\colon \cY(\bZ_p) \to \bQ_p$, $\rho(P) = \int_{P_0}^P \omega$ takes a constant value $c(P_0,\Sigma,\omega)$ on all integral points of a fixed reduction type~$\Sigma$. By the same argument as above, the intersection numbers in~\eqref{eq:split-hyperelliptic-constants-final} vanish, so the constants $c(P_0,\Sigma,\omega)$ are all zero. The function~$\rho$ thus vanishes on all of~$\cY(\bZ)$. In particular, it is not necessary to distinguish between points of different reduction types.
We can analyse the zeros of~$\rho$ on each residue disc by computing the power series expansion in a uniformising parameter. For example, consider the point $P_0 = (-1,1)$. Its residue disc is parametrised by $Q(t) = (x(t), y(t))$ with
\begin{align*}
	x(t) = -1 + 7t,\quad 
	y(t) = \sqrt{f(x(t))} = \sum_{n=0}^\infty \binom{1/2}{n}(f(-1 + 7t)-1)^n %
\end{align*}
and in terms of the parameter~$t$, the function~$\rho$ is given by
\begin{align*}
	\rho(Q(t)) = \sum_{j=0}^2 a_j \int_{P_0}^{Q(t)} x(t)^j \frac{\rd x(t)}{y(t)} = (7 + 3\cdot 7^2) t + 6\cdot 7^2 \cdot t^2 + O(7^3,t^3)
\end{align*}
The sequence of valuations of the coefficients starts $(\infty, 1, 2, 3, 4, 6, 7, \ldots)$, so by Strassmann's Theorem (see for example \cite[Theorem~5.6.1]{gouvea}), the only zero of $\rho(Q(t))$ in~$\bZ_7$ is at $t = 0$. This shows that $P_0 = (-1,1)$ is the only integral point of~$\cY$ in its residue disc. Each of the remaining residue discs contains one of the known integral points, and in each case one can show with Strassmann's Theorem that there are no other integral points in that disc, proving \Cref{thm:example-split}. Sage code for computing the vanishing locus of the function~$\rho$ in this example, for any auxiliary prime~$p$, can be found at \url{https://github.com/martinluedtke/AffChab2}.

\subsection{A class of superelliptic curves}
\label{sec:superelliptic}
We now discuss an example where our method can compute $S$-integral points for nonempty $S$.
Namely, we aim to determine the $S$-integral points on affine curves of the form
\begin{equation}
	\label{ex:supercurves}
	\cY\colon \quad y^3 = x^3 + ax^2 + x,
\end{equation}
where $a \in \bZ\setminus\{\pm 2\}$ is a fixed parameter.
Let $\cX$ be the closure of~$\cY$ in $\bP^2_{\bZ}$, let $\cD = \cX \smallsetminus \cY$, and denote by $X,Y,D$ the generic fibres. The curve~$X$ is the smooth compactification
of~$Y$. 
We always have the rational point $P_0 = (0,0) \in Y(\bQ)$, which we use as base point.
The divisor of cusps~$D$ consists of the rational point $Q_1 = (1 : 1: 0)$ and the degree-2 point $Q_2 = \{(1 : \zeta_3 : 0), (1 :\zeta_3^{-1} : 0)\}$, where $\zeta_3$ is a primitive third root of unity.
Setting $(u,v) \coloneqq (y/x, 1/x)$, an affine chart containing the cusps is given by $v^2 + av = u^3 - 1$, a Weierstrass equation for an elliptic curve. Note that this change of coordinates does not preserve $S$-integrality, so the problem of finding $S$-integral points of~$\cY$ is different from finding $S$-integral points on an elliptic curve. (This is why we refer to~\eqref{ex:supercurves} as a superelliptic equation even though it defines a genus~1 curve.) 
We see that $n_1(D) = n_2(D) = 1$, $n = 3$, and $g = 1$. Let $r$ be the rank of $J(\bQ) = \Jac_X(\bQ)$. Assumption \eqref{eq:chabauty-condition-over-number-field-intro} is satisfied whenever $r+\#S\leq 2$. We focus on the most interesting case when $r = \#S = 1$. The model~$\cX \subseteq \bP^2_{\bZ}$ has singular fibres over~$3$ and the prime divisors of $a \pm 2$. In order to apply the Affine Chabauty method for a given~$a$, one first needs to compute a regular model by performing blow-ups in the singular points if necessary. One then needs to compute intersection numbers and find a Mordell--Weil basis, i.e.\ perform \ref{item:algo-D-transversal-model}--\ref{item:algo-correction-divisors} of \Cref{sec:algorithm}. Before giving an explicit example, let us address steps \ref{item:algo-log-differentials} and \ref{item:algo-integrals}.

The log differentials
\[ \omega_1 = \frac{\rd x}{y^2} = -\frac{3 \rd u}{2v + a} ,\quad \omega_2 = \frac{x \,\rd x}{y^2} = -\frac{3 \rd u}{v(2v+a)} , \quad \omega_3 = \frac{\rd x}{y} = -\frac{3u \rd u}{v(2v+a)} \]
form a basis of $\rH^0(X, \Omega^1(D))$. Their residues $\Res_{Q_i}(\omega_j)$ are
\[
	\begin{array}{c|ccc}
		& \omega_1 & \omega_2 & \omega_3 \\
		\hline
		\rule{0pt}{2.6ex} %
		Q_1 & 0 & -1 & -1 \\[1mm]
		Q_2 & 0 & -\zeta_3 & -\zeta_3^{-1}
	\end{array}
\]
Here, $\zeta_3 =u(Q_2) \in k(Q_2)$ is a primitive third root of unity in the residue field of~$Q_2$. Fix an auxiliary prime $p \equiv 1 \pmod 3$ where $\cX$ has good reduction so that $\bQ_p$ contains a primitive third root of unity. Let us explain how to compute the Coleman integrals $\int_P^Q \omega_j$ of the basis differentials. For this it is convenient to work with the short Weierstrass equation $v'^2 = u'^3 + a^2/4 - 1$ in the coordinates $(u',v') \coloneqq (u,v + a/2)$. One has $\omega_1 = -3 \frac{\rd u'}{2v'}$, which is an invariant differential and can be integrated using the function \texttt{coleman\_integrals\_on\_basis} from \url{https://github.com/jbalakrishnan/AWS}. For the non-holomorphic differentials $\omega_2$ and $\omega_3$ we write $\omega_j = \omega_j^+ + \omega_j^-$, where $\omega_j^{\pm} = (\omega_j \pm \iota^* \omega_j)/2$ are the symmetric and anti-symmetric component of $\omega_j$ under the hyperelliptic involution $\iota\colon (u',v') \mapsto (u', -v')$, respectively. The symmetric differentials are pulled back from $\bP^1$ via the hyperelliptic morphism $u'\colon X \to \bP^1$. Switching to the coordinate $w \coloneqq 1/u'$ and using partial fraction decomposition one finds
\begin{align*}
	\omega_2^+ &= -\frac{3}{2} \frac{\rd u'}{u'^3 - 1} = -\frac12 \sum_{\zeta^3 = 1} \frac{\zeta^{-1}}{w - \zeta} \rd w,\\
	\omega_3^+ &= -\frac32 \frac{u' \, \rd u'}{u'^3 - 1} = -\frac12 \sum_{\zeta^3 = 1} \frac{\zeta}{w - \zeta} \rd w.
\end{align*}
We can compute these integrals on $\bP^1$ using $\int_P^Q \frac{\rd w}{w - \zeta} = \log(\frac{w(Q) - \zeta}{w(P) - \zeta})$. For the antisymmetric differentials we find
\begin{align*}
	\omega_2^- &= -\frac{3}{4v'} \frac{a}{u'^3 - 1} \rd u' = -\sum_{\zeta^3 = 1} \zeta \frac{a/2}{u' - \zeta} \frac{\rd u'}{2v'},\\
	\omega_3^- &= -\frac{3u'}{4v'} \frac{a}{u'^3 - 1} \rd u' = -\sum_{\zeta^3 = 1} \zeta^{-1} \frac{a/2}{u' - \zeta} \frac{\rd u'}{2v'}.
\end{align*}
The integrals $\int_P^Q \frac{a/2}{u' - \zeta} \frac{\rd u'}{2v'}$ can be computed using a trick from \cite[(6.2)]{GM:computing-p-adic-heights}. Let $X_{\zeta}/\bQ_p$ be the even-degree genus-1 curve defined by $t^2 = \frac{4}{a^2}s^4((\zeta + s^{-1})^3 + a^2/4 - 1)$. The map $\tau_{\zeta}\colon X \to X_{\zeta}$, $(u',v') \mapsto (1/(u' - \zeta), -2v'/(a(u'-\zeta)^2))$ sends the points $(u',v') = (\zeta,\pm a/2)$ to the points at infinity of $X_{\zeta}$ and satisfies 
\[ \tau_{\zeta}^*\left(\frac{s \, \rd s}{2t}\right) = \frac{a/2}{u' - \zeta} \frac{\rd u'}{2v'}, \]
so the integrals can be computed on~$X_{\zeta}$ as
\[ \int_P^Q \frac{a/2}{u' - \zeta}\frac{\rd u'}{2v'} = \int_{\tau_{\zeta}(P)}^{\tau_{\zeta}(Q)} \frac{s \, \rd s}{2t}. \]
The integral of this particular differential form can be computed by the SageMath function \texttt{coleman\_integrals\_on\_basis}, too. %
The code cannot currently handle endpoints in infinite residue discs, which means that on the original curve the endpoints $P$ and $Q$ need to lie outside of the residue discs of the cusps $Q_1$, $Q_2$ and their images $\iota(Q_1)$, $\iota(Q_2)$ under the hyperelliptic involution. 

\begin{example}\label{exa:superell}
	We specialise to the case $a=1$.
	Now $X$ has Mordell--Weil rank $r=1$ and the model $\cX \subseteq \bP^2_{\bZ}$ given by equation \eqref{ex:supercurves} is regular, only has bad reduction at $3$ and every fibre has only one component. The closures $\cQ_1$ and $\cQ_2$ of the cusps $Q_1$ and $Q_2$ in~$\cX$ are normal but they intersect in the mod-3 fibre in the point $(\overline{1} \colon \overline{1} \colon \overline{0}) \in \cX_3(\bF_3)$, so the model is not $D$-transversal over~$3$. We choose $S = \{487\}$, where a non-obvious $S$-integral point is given by $(x,y)=(\tfrac{216}{487},\tfrac{438}{487})$, and demonstrate our method for the reduction type $\Sigma = (\Sigma_q)_q$ where $\Sigma_q$ is the unique component of $\cX_q^{\sm} \smallsetminus\cD_q$ for $q \neq 487$, and $\Sigma_{487} \in \cX(\bF_{487})$ is the point on~$\cQ_2 \cong \Spec(\bZ[\zeta_3])$ defined by the prime ideal $(2\zeta_3 + 23)$, or equivalently the point $(\overline{1} : \overline{232} : \overline{0}) \in \cX(\bF_{487})$.
	The result is a Coleman function that vanishes on all $\{487\}$-integral points of reduction type~$\Sigma$. We compute its vanishing locus on all residue discs of $\cY(\bZ_p)$ on which we can compute the Coleman integrals (those with $x \not\equiv -1 \bmod p$), and confirm that it indeed vanishes on the $\{487\}$-integral point $(\tfrac{216}{487},\tfrac{438}{487})$ of reduction type~$\Sigma$. 
	We follow the steps in \Cref{sec:explicit}.
	\begin{enumerate}[label=(\arabic*)]
		\item $G=A-P_0$ with $A = (1/18, 7/18)$ and $P_0 = (0,0)$ (in the coordinates $(x,y)$) generates an index $2$ subgroup of $J(\bQ)$.
		\item For all primes $q$, the special fibre $\cX_q$ has only one component, so we take $\Phi_q(G)=0$.
		\item We use $\omega_j$ as above. %
		
		\item As the residue fields of both cusps have unit rank $0$, this step is irrelevant.
		\item \label{item:b=zero} We use $\fS_{\bQ}(P_0,\Sigma) = b + U(\Sigma^{\csp})$.
		We calculate $b=0$ using \eqref{eq:b-lambda}.
		Indeed, the first summand vanishes as $P_0$ never reduces to a cusp; the second summand vanishes as $\cpt(\Sigma_q) = \cpt_q(P_0)$ for all primes $q$.
		A basis of $U(\Sigma^{\csp})$ is given by $u = [\Sigma_{487}]$.
		\item We choose $q$ as the generator of the prime ideal $q\bZ$ for every $q$.
		\item We take $\pi_2 = 2$ and $\pi_3 = 3$ as prime ideal generators in $k(Q_1)=\bQ$; for $k(Q_2)=\bQ(\zeta_3)$, we take $\pi_{(\sqrt{-3})}=\sqrt{-3}\coloneqq 1+2\zeta_3$ and $\pi_{(2\zeta_{3}+23)} = 2\zeta_3 + 23$. 
	\end{enumerate}
			
	Choose an auxiliary prime $p \equiv 1 \bmod 3$, a primitive third root of unity $\zeta_3\in \bQ_p$ and put $\sqrt{-3}\coloneqq 1+2\zeta_3\in\bQ_p$.
	Then the matrix $M(\Sigma^{\csp})$ of \eqref{eq:MSigma} is equal to
	\begin{equation*}
		M(\Sigma^{\csp}) = \begin{pmatrix}
			\int_{P_0}^A \omega_1 & \int_{P_0}^A \omega_2 - \beta_2 & \int_{P_0}^A \omega_3 - \beta_3\\
			0 & D(\Sigma^{\csp})_2 & D(\Sigma^{\csp})_3
			\end{pmatrix},
		\end{equation*}
	where using \eqref{eq:H} we have
	\begin{align*}
			\beta_2 &= 
(-1) \left(i_{\pi_2}(\cA - \cP_0, \cQ_1) \log(2) + i_{\pi_3}(\cA - \cP_0, \cQ_1) \log(3)\right) \\
					&\quad -\zeta_3 i_{\pi_{(\sqrt{-3})}}(\cA - \cP_0, \cQ_2) \log(\sqrt{-3}) -\zeta_3^{-1} i_{\pi_{(\sqrt{-3})}}(\cA - \cP_0, \cQ_2) \log(-\sqrt{-3}) \\
					&= -\log(2) - \frac12 \log(3)
		\end{align*}
	because $A$ reduces onto $\cQ_1$ modulo $2$ and onto $\cQ_1 \cap \cQ_2$ modulo $3$; $P_0$ never reduces to a cusp; $i_{\pi_\bullet}(\cA,\cQ_i)=1$ in all cases; and $\log(\sqrt{-3})=\log(3)/2$.
	A similar calculation shows $\beta_3 = \beta_2$.
	The entries $D(\Sigma^{\csp})_i$ are computed using \eqref{eq:matrixD}:
	\begin{align*}
			D(\Sigma^{\csp})_2 &= %
			-\zeta_3 \log(2\zeta_3 +23) - \zeta_3^{-1} \log(2\zeta_3^{-1} + 23),\\
			D(\Sigma^{\csp})_3 &= -\zeta_3^{-1} \log(2\zeta_3 +23) - \zeta_3 \log(2\zeta_3^{-1} + 23).
		\end{align*}
	We compute the integrals $\int_{P_0}^A \omega_j$ with the method outlined above. Taking $p = 7$, the smallest possible choice, we obtain the matrix
	\begin{equation*}
		M(\Sigma^{\csp})= %
			\begin{pmatrix}
			2\cdot 7 + 5\cdot 7^2 + 4 \cdot 7^4 & 6\cdot 7 + 7^2 + 3\cdot 7^4 & 2\cdot 7 + 6\cdot 7^2 + 2\cdot 7^3  \\ %
			0 & 6\cdot 7^2 + 2\cdot 7^3 + 6\cdot 7^4  & 7 + 4\cdot 7^2 + 7^3 + 5\cdot 7^4 %
			\end{pmatrix}
		\end{equation*}
	displayed with precision $O(7^5)$.
	The element
	\[
	(a_1,a_2,a_3) = (1, 2 + 6\cdot 7 + 2\cdot 7^2 + 3\cdot 7^3 , 2\cdot 7 + 6\cdot 7^2)  + O(7^5) %
	\]
	lies in the kernel of $M(\Sigma^{\csp})$, so the differential $\omega = a_1\omega_1 + a_2 \omega_2 + a_3\omega_3$ vanishes on all $\{487\}$-integral points of reduction type $\Sigma$; note that the constant $c(P_0,\Sigma,\omega)$ is zero by \eqref{eq:constants} and $b=0$ by~\ref{item:b=zero} above.
	We computed the complete vanishing locus of the function $P \mapsto \int_{P_0}^P \omega$ on residue discs with $x \not\equiv -1 \bmod 7$. We did the same for the remaining two cuspidal reduction types with $\Sigma_{487} = (\overline{1} : \overline{254} : \overline{0}) \in \cQ_2(\bF_{487})$ and $\Sigma_{487} = (\overline{1} : \overline{1} : \overline{0}) \in \cQ_1(\bF_{487})$, respectively. The results are shown below:
	\begin{center}
		{\renewcommand{\arraystretch}{1.15}
		\begin{tabular}{|c||c|}
			\hline
			reduction type & Chabauty locus\footnote{in residue discs with $x \not \equiv -1 \bmod 7$} \\
			\hline \hline
			\rule{0pt}{2.5ex} $(\overline{1} : \overline{232} : \overline{0}) \bmod 487$ &
				$(0,0)$, $(\zeta_3,0)$, $(\zeta_3^{-1},0)$, $(\tfrac{216}{487},\tfrac{438}{487})$, \\
			&  
				$(5 + 3\cdot 7 + 4\cdot 7^2 + 7^3 + 6\cdot 7^4 + O(7^5), 2 + O(7))$,\\ %
			&	$(5 + 3\cdot 7 + 6\cdot 7^2 + 4\cdot 7^3 + 4 \cdot 7^4 + O(7^5), 4 + O(7))$ \\ %
			\hline
			\rule{0pt}{2.5ex} $(\overline{1} : \overline{254} : \overline{0}) \bmod 487$ & $(0,0)$, $(\zeta_3,0)$, $(\zeta_3^{-1},0)$, \\
			& $(5 + 4\cdot 7 + 3\cdot 7^2 + 5\cdot 7^3 + O(7^5), 2 + O(7))$,\\ %
			& $(5 + 6\cdot 7^2 + 4\cdot 7^3 + 3\cdot 7^4 + O(7^5), 4 + O(7))$ %
			\\ 
			\hline
			 \rule{0pt}{2.5ex} $(\overline{1} : \overline{1} : \overline{0}) \bmod 487$ & 
			$(0,0)$, $(\zeta_3,0)$, $(\zeta_3^{-1},0)$, $(\tfrac1{18},\tfrac7{18})$, $\bigl(\frac{-2-\sqrt{-3}}{7}, \frac{3+5\sqrt{-3}}{14}\bigr)$, \\
			& $(5 + 2\cdot 7 + 2\cdot 7^3 + 7^4 + O(7^5), 2 + O(7))$,\\ %
			& $(5 + 7 + 5\cdot 7^2 + 5\cdot 7^3 + O(7^5), 4 + O(7))$ \\ %
			\hline
		\end{tabular}}
	\end{center}
	Note that the $\{487\}$-integral point $(\tfrac{216}{487},\tfrac{438}{487})$ is correctly picked up by our Chabauty function. The points $(0,0)$, $(\zeta_3,0)$, $(\zeta_3^{-1},0)$ show up in the locus since they are fixed points of the automorphism $(x,y) \mapsto (x,\zeta_3 y)$ of $Y_{\bQ_p}$, and it follows from the change-of-variables formula that Coleman integrals of log differentials between any two of these points vanish. The appearance of the rational but not $S$-integral point $(\tfrac1{18},\tfrac7{18})$ and of the algebraic but not rational point $(\tfrac{-2-\sqrt{-3}}{7}, \tfrac{3+5\sqrt{-3}}{14})$ is slightly surprising but not an uncommon phenomenon in Chabauty methods. The remaining points appear to be transcendental. With some additional work, using Mordell--Weil sieve techniques as in \cite{BBM:integral-points}, one could try to show that these points are not rational but we have not attempted this.
	The code for this example can be found on \url{https://github.com/martinluedtke/AffChab2}.
\end{example}

\printbibliography

\end{document}